\documentclass[a4paper,12pt]{amsart}%
\usepackage{amssymb}
\usepackage{amsmath}
\usepackage{amsfonts}
\usepackage{doublespace}
\usepackage{color}
\usepackage{graphicx}%
\providecommand{\U}[1]{\protect\rule{.1in}{.1in}}
\theoremstyle{plain}
\newtheorem{theorem}{Theorem}

\newtheorem{lemma}[theorem]{Lemma}

\numberwithin{equation}{section}
\newtheorem{theoremUN}{Theorem}

\begin{document}
\title[Approximations by canonical almost embeddings]{Bishop-Runge approximations and inversion of a Riemann-Klein theorem}
\author{Gennadi Henkin}
\address[G. Henkin]{Universit\'{e} Pierre et Marie Curie, Paris, France; CEMI, Academy
of Science, Moscow, Russi\vspace{-0.2cm}a\vspace{-0.2cm}}
\author{Vincent Michel}
\address[V. Michel]{Universit\'{e} Pierre et Marie Curie, Paris, France}
\email{henkin@math.jussieu.fr, michel@math.jussieu.fr}
\date{20 Novembre 2013}

\begin{abstract}%
\begin{spacing}{1.1}%
In this paper we give results about projective embeddings of Riemann surfaces,
smooth or nodal, which we apply to the inverse Dirichlet-to-Neumann problem
and to the inversion of a Riemann-Klein theorem. To produce useful embeddings,
we adapt a technique of Bishop in the open bordered case and use Runge type
harmonic approximation theorem in the compact case.

%

\end{spacing}%

\end{abstract}
\maketitle

%

\begin{spacing}{1.1}%

\section{Introduction}

Applied to an open bordered Riemann surface, the works of
Bishop~\cite{BiE1961} and Narasimhan~\cite{NaR1960} about the embedability of
Stein manifolds imply the following\medskip

\begin{theoremUN}
[Bishop-Narasimhan]Let $Z$\textit{\ }be an open bordered Riemann surface,
$n\geqslant2$\ and\textit{\ }$f:Z\longrightarrow\mathbb{C}^{n}$\textit{\ }a
holomorphic map smooth up to the boundary. Then, $f$ can be uniformly
approximated on $\overline{Z}$\ by embeddings if $n\geqslant3$ and if $n=2$,
by immersions injective outside of a finite set.
\end{theoremUN}

That this holds when $\mathbb{C}^{n}$ is replaced by the complex projective
space $\mathbb{CP}_{n}$ of the same dimension is a question arising in
applications but seems to haven't yet been addressed. A natural way to
construct projective maps of a Riemann surface $Z$ is to use global sections
of its canonical bundle $K\left(  Z\right)  $. Indeed, if $\omega=\left(
\omega_{0},...,\omega_{n}\right)  \in K\left(  Z\right)  ^{n+1}$ never
vanishes, $\omega$ induces a \textit{canonical map }denoted $\left[
\omega\right]  $ or $\left[  \omega_{0}:\cdots:\omega_{n}\right]  $ from $Z$
to $\mathbb{CP}_{n}$ defined by the formulas
\begin{equation}
\left[  \omega\right]  =\left[  \frac{\omega_{0}}{\omega_{j}}:\cdots
:\frac{\omega_{j-1}}{\omega_{j}}:1:\frac{\omega_{j+1}}{\omega_{j}}%
:\cdots:\frac{\omega_{n}}{\omega_{j}}\right]  \label{F app}%
\end{equation}
on $\left\{  \omega_{j}\neq0\right\}  $, $0\leqslant j\leqslant n$. As in the
affine case, there is an approximation statement when the target space is
projective.\smallskip

\noindent\textbf{Theorem A. }\label{T/ A}\textit{Let }$Z$\textit{\ be an open
bordered Riemann surface, }$n\geqslant2$\textit{, and }$f:\overline
{Z}\longrightarrow\mathbb{CP}_{n}$\textit{\ a holomorphic map smooth up to the
boundary. Then }$f$\textit{\ can be uniformly approximated on }$\overline{Z}%
$\textit{\ by canonical embeddings if }$n\geqslant3$\textit{\ and if }$n=2$,
\textit{by canonical immersions injective outside of a finite set.}\smallskip

This result is a consequence of theorem~\ref{T AEB2}, which is a variation of
the density theorem~\ref{T AEB1} needed in~\cite{HeG-MiV2007, HeG-MiV2013} for
our constructive solution of the inverse Dirichlet-to-Neumann problem for
Riemann surfaces.

This problem can be stated in the constant conductivity case as follows.
Consider in $\mathbb{R}^{3}$ an open bordered real surface $Z$ endowed with
the complex structure induced by the standard euclidean structure of
$\mathbb{R}^{3}$~; every Riemann surface admits such a presentation according
to Garsia~\cite{GaA1961} for the compact case and to R\"{u}edy~\cite{RuR1971}
for the bordered case. Assume that $Z$ has constant conductivity. Then the
Dirichlet-to-Neumann operator of $Z$ can be seen as the operator $N$ which to
$u\in C^{1}\left(  bZ\right)  $ associates the boundary current induced by the
electrical potential $\widetilde{u}$ created by $u$ in $Z$, that is
$Nu=\left(  d^{c}\widetilde{u}\right)  \left\vert _{bZ}\right.  $ where
$\widetilde{u}$ is the harmonic extension to $Z$ of $u$ and $d^{c}=i\left(
\overline{\partial}-\partial\right)  $, $\overline{\partial}$ being the
Cauchy-Riemann operator of $Z$. The \textit{inverse Dirichlet-to-Neumann
problem} in this case is to reconstruct $Z$ from a finite number of
measurements of boundary currents, that is from $\left(  Nu_{\ell}\right)
_{0\leqslant\ell\leqslant n}$ where the boundary potentials $u_{\ell}$,
$0\leqslant\ell\leqslant n$, are known. Canonical maps appear here naturally
since data of $\left(  u_{\ell},Nu_{\ell}\right)  _{0\leqslant\ell\leqslant
n}$ and $\left(  u_{\ell},\left(  \partial\widetilde{u_{\ell}}\right)
\left\vert _{bZ}\right.  \right)  _{0\leqslant\ell\leqslant n}$ are
equivalent. A recent interesting survey on this topic is written
in~\cite{GuCTzL2012}.

In~\cite{HeG-MiV2007}, we study this problem under its three aspects~:
uniqueness, reconstruction and characterization. We prove in particular that
if the boundary of $Z$ is known, the knowledge of only three boundary
potentials $u_{\ell}$, $0\leqslant\ell\leqslant2$, and their associated
boundary currents $\left(  \partial\widetilde{u_{\ell}}\right)  \left\vert
_{bZ}\right.  $, $0\leqslant\ell\leqslant2$, is sufficient to recover $Z$ with
Cauchy type integral formulas from the canonical boundary map $\left[  \left(
\partial\widetilde{u}\right)  \left\vert _{bZ}\right.  \right]  $ when
$\left(  u_{\ell}\right)  _{0\leqslant\ell\leqslant2}$ fulfills a generic
hypothesis which was not quite correctly formulated in~\cite{HeG-MiV2007} and
rightly in \cite{HeG-MiV2013} where some of our results are extended to open
bordered nodal curves with electrically charged nodes.

In the present paper, we prove that data satisfying the assumptions
of~\cite{HeG-MiV2013} are generic~; these assumptions essentially mean that
the canonical map $\left[  \partial\widetilde{u}\right]  $ is an immersion
which embeds $bX$ into $\mathbb{CP}_{2}$ and is proper in the sense that the
pullback of the image of $bX$ reduces to $bX$~; such data are sufficient for
the results of~\cite{HeG-MiV2007, HeG-MiV2013} to hold.\smallskip

Our results in the inverse Dirichlet-to-Neumann problem are also applied
in~\cite{HeG-MiV2013} to obtain an inversion of the following theorem of
Riemann and Klein.\vspace{0.2cm}

\begin{theoremUN}
[Riemann 1857, Klein 1882]Let $Z$\ be a compact Riemann surface, $\left(
a_{j}^{\pm}\right)  _{1\leqslant j\leqslant\nu}$ a\ family of mutually
distinct points of $Z$\ and $\left(  c_{j}\right)  _{1\leqslant j\leqslant\nu
}\in\mathbb{R}^{\nu}$. Then, there exists a unique (up to an additive
constant) harmonic function $U$ on $Z\backslash\left\{  a_{j}^{\pm
};~1\leqslant j\leqslant\nu\right\}  $ with at most logarithmic singularities
such that the residue $\operatorname*{Res}_{a_{j}^{\pm}}\left(  d^{c}U\right)
\overset{def}{=}\frac{1}{2\pi}\int_{dist(.,a_{j}^{\pm})=\varepsilon}d^{c}U$
($\varepsilon>0$\ small enough) of $d^{c}U=i\left(  \overline{\partial
}-\partial\right)  U$\ at $a_{j}^{\pm}$\ is $\pm c_{j}$, $1\leqslant
j\leqslant\nu$.
\end{theoremUN}

In this statement $Z$ should be seen as a real compact surface of
$\mathbb{R}^{3}$ endowed with the complex structure induced by the standard
euclidean structure of $\mathbb{R}^{3}$ with constant conductivity and the
couples $\left(  a_{j}^{-},a_{j}^{+}\right)  $ as dipoles electrically charged
with $\left(  -c_{j},+c_{j}\right)  $. In this setting, an\textit{\ electrical
potential} is a harmonic function having only a finite number of logarithmic
isolated singularities whose \textit{charges}, that is residues in the above
sense, have a vanishing sum~; the sum of their modulus is called the $L^{1}$
norm of this\textit{\ distribution of charges}.

Our paper \cite{HeG-MiV2013} contains the result below about the inversion of
a theorem of Riemann and Klein which can be formulated as the reconstruction
of a compact Riemann surface from data collected on a (small) known
subdomain.\smallskip

\noindent\textbf{Theorem B. }\label{T/ B}\textit{Let }$S$\textit{\ be an open
subdomain of a compact Riemann surface }$Z$ \textit{considered as a
submanifold of }$\mathbb{R}^{3}$ \textit{equipped with the conformal structure
induced by the standard euclidean structure on} $\mathbb{R}^{3}$\textit{. Let
}$\left(  U_{\ell}\right)  _{0\leqslant\ell\leqslant2}$\textit{\ be a family
of potentials induced on }$Z$ \textit{by electrically charged dipoles in }%
$S$\textit{. Then }$Z\backslash S$\textit{\ can be reconstructed from
}$\left(  U_{\ell}\left\vert _{S}\right.  \right)  _{0\leqslant\ell\leqslant
2}$ \textit{or} $\left(  V_{\ell}\left\vert _{S}\right.  \right)
_{0\leqslant\ell\leqslant2}$ \textit{where} $\left(  V_{\ell}\right)
_{0\leqslant\ell\leqslant2}$ \textit{is a family} \textit{of electrical
potentials\ such that }$\left(  U_{\ell}-V_{\ell}\right)  _{0\leqslant
\ell\leqslant2}$ \textit{is arbitrarily closed to }$0 $\textit{\ in
}$C^{\infty}\left(  Z\backslash S\right)  ^{3}$ \textit{and} \textit{is
induced by a distribution of electrical charges confined in }$S$\textit{\ with
arbitrarily small }$L^{1}$\textit{\ norm.}\smallskip

This result is a consequence of theorem~\ref{T/ compact generique} below. It
is interesting to compare theorem~B with the results of H.
Stahl~\cite{StH1997} and of A. Gonchar, E. Rakhmanov, S.
Suetin~\cite{GoA-RaE-SuS2011} on Pade approximations of algebraic functions.
In spite of that our technique is very different from method
of~\cite{GoA-RaE-SuS2011} and [\cite{StH1997}], our purpose is quite similar~:
constructive reconstruction of an algebraic Riemann surface $Z$ from optimal
electrostatic data collected in a neighborhood $S$ of one point in $Z$.

To prove the perturbation statement of theorem~B, we need a quantitative
version of known~\cite{BaT1988, BoAGaP1984} qualitative approximation theorems
for harmonic functions. This is done with theorem~\ref{T/ Runge}~; its proof
uses a discretization of the Hodge-de~Rham decomposition
formula~(\ref{F/ Hodge}).\smallskip

The next section is devoted to the statement of our mains results which
contain theorems~A and~B. Proofs are in the last two sections.

\section{Statements of theorems}

An \textit{open bordered Riemann surface} is the interior of a one dimensional
compact complex manifold with boundary whose all connected components have non
trivial one real dimensional smooth boundary.

An \textit{open (bordered) nodal curve} is a quotient of an open
(bordered)\textit{\ }Riemann surface $Z$ by an equivalence relation
identifying a finite number of interior points. If $X$ is a nodal curve, a
\textit{branch} of $X$ is any connected Riemann surface contained in $X$. A
\textit{node }of $X$ is a point of the singular locus $\operatorname*{Sing}X$
of $X$~; the number of germs of branches of $X$ passing through $a$ is denoted
by $\nu\left(  a\right)  $.

If $G$ is an open subset of $X$ and $r\in\left[  0,+\infty\right]  $, a
function $u$ on $G$ is said to be of class $C^{r}$ if it is continuous and if
for any branch $B$ of $G$, $u\left\vert _{B}\right.  \in C^{r}\left(
B\right)  $~; the space of such functions is denoted by $C^{r}\left(
G\right)  $ or $C_{0,0}^{r}\left(  G\right)  $. If $p,q\in\left\{
0,1\right\}  $ and $p+q>0$, a $\left(  p,q\right)  $-form $\omega$ of
$C_{p,q}^{r}\left(  G\cap\operatorname*{Reg}X\right)  $ is said to be of class
$C^{r}$ on $G$ if for any branch $B$ of $G$, $\omega\left\vert _{B\cap
\operatorname*{Reg}X}\right.  $ extends as an element of $C_{p,q}^{r}\left(
B\right)  $. The space of such forms is denoted by $C_{p,q}^{r}\left(
G\right)  $.

If $K$ is a compact subset of $X$ and $p,q\in\left\{  0,1\right\}  $, the
space $C_{p,q}^{\infty}\left(  K\right)  $ of smooth $\left(  p,q\right)
$-forms supported in $K$ is equipped with the topology induced by the
semi-norms $\left\Vert \omega\right\Vert _{m,K,B}=~\underset{_{B\cap K}}{\sup
}\left\Vert D^{\left(  m\right)  }\omega\left\vert _{B}\right.  \right\Vert $
where $m$ is any integer, $B$ any branch of $X$ and $D$ is the total
differential acting on coefficients~; if $K\subset\operatorname*{Reg}X$, the
index $B$ is omitted.

The space $D_{p,q}\left(  G\right)  $ of smooth $\left(  p,q\right)  $-forms
compactly supported in $G$ is equipped with the inductive limit topology of
the spaces $C_{p,q}^{\infty}\left(  K\right)  $ where $K$ is any compact of $G
$. The space $D_{p,q}^{\prime}\left(  G\right)  $ of \textit{currents} on $G $
of bidegree $\left(  p,q\right)  $ is the topological dual of $D_{p,q}\left(
G\right)  ~$; the elements of $D_{1,1}^{\prime}\left(  G\right)  $ are the
\textit{distributions} on $G$. The exterior differentiation $d$ of smooths
forms is well defined along branches of $X$, so it is for $\partial$ and
$\overline{\partial}$. These operators extend to currents by duality.

A \textit{harmonic} \textit{distribution }is an element $U$ of $D_{1,1}%
^{\prime}\left(  G\right)  $ which is \textit{(weakly) harmonic} in the sense
that the current $i\partial\overline{\partial}U$ vanish, that is $\left\langle
i\partial\overline{\partial}U,\varphi\right\rangle =0$ for all $\varphi\in
C_{c}^{\infty}\left(  G\right)  $. Thus, $U$ is harmonic if and only if
$\partial U$\ is a $\overline{\partial}$-closed current. Note that $\omega\in
D_{0,1}^{\prime}\left(  G\right)  $ is $\overline{\partial}$-closed if and
only if $\omega$ is a weakly holomorphic $\left(  1,0\right)  $-form in the
sense of Rosenlicht~\cite{RoM1954}.

According to~\cite[prop.~2]{HeG-MiV2013}, a harmonic distribution $U$ on $G$
is a usual harmonic function on $\operatorname*{Reg}G$ such that for any node
$a$ of $X$ lying in $G$ and any branch $B$ of $X$ at $a$, $U\left\vert _{B\cap
G}\right.  $ has at most a \textit{logarithmic isolated singularity}, that is
$U\left\vert _{B\cap G}\right.  =\operatorname*{Res}_{B}\left(  U,a\right)
\ln\operatorname{dist}\left(  .,a\right)  +R$ where $\operatorname*{Res}%
_{B}\left(  U,a\right)  \in\mathbb{C}$, the distance is computed in any
hermitian metric of $B$ and the remainder $R$ is smooth near $a$ in $B$~; in
addition, the sum of the complex numbers%
\[
\operatorname*{Res}{}_{B}\left(  U,a\right)  =\operatorname*{Res}\left(
\left(  \partial U\right)  \left\vert _{B}\right.  ,a\right)  =\frac{1}{2\pi
i}\int_{B\cap\left\{  dist\left(  .,a\right)  =\varepsilon\right\}  }\partial
U
\]
($\varepsilon>0$ small enough) taken over the branches $B$ of $X$ at $a$ is
$0$.

For any given $u\in C^{\infty}\left(  bX\right)  $ and \textit{admissible
family }$c$, that is a family $\left(  c_{a,j}\right)  _{a\in
\operatorname*{Sing}X,~1\leqslant j\leqslant\nu\left(  a\right)  }$ of complex
numbers such that $\underset{1\leqslant j\leqslant\nu\left(  a\right)
}{\Sigma}c_{a,j}=0$ for each $a\in\operatorname*{Sing}X$, there is a unique
harmonic distribution extension $\widetilde{u}^{c}$ of $u$ to $X$ with $c$ as
family of residues (see e.g. \cite[prop.~2]{HeG-MiV2013})~; when $X$ is
smooth, the only admissible family is the empty one and $\widetilde{u}%
^{\varnothing}=\widetilde{u}$ is the usual harmonic extension of $u$ to $X$.

Weakly holomorphic $\left(  1,0\right)  $-forms on $G$ are usual holomorphic
forms on $\operatorname*{Reg}G$, meromorphic on branches of $G$ and at each
pole, the sum of residues along branches passing through it is equal to $0$.
In particular, they are entitled to produce canonical maps. It is worth
noticing that the singularities of weakly holomorphic $\left(  1,0\right)
$-forms on a nodal curve are at most simple poles and that at each pole the
sum of the residues is $0$. These facts are described in algebraic language
in~\cite{HaJ-MoIL1998} for general curves~; in the case of nodal curves,
\cite[prop.~2]{HeG-MiV2013} gives an elementary justification.

An \textit{immersion} of an open bordered nodal curve $X$ in some complex
manifold $M$ is a map multivaluate from $\overline{X}$ to $M$ such that its
restriction to any branch of $X$ is a usual univaluate immersion.

An \textit{almost embedding} is an immersion $\varphi:\overline{X}%
\longrightarrow M$ with the following properties~: $X^{\prime}=\varphi\left(
X\right)  $ is an analytic subset of $M\backslash bX^{\prime}$~; there exists
a finite subset of $E$ of $X$ such that $\varphi\left(  E\right)
\subset\varphi\left(  X\right)  \backslash\varphi\left(  bX\right)  $ and
$\varphi$ is an isomorphism of Riemann surfaces from $\left(
\operatorname*{Reg}X\right)  \backslash E$ onto $\left(  \operatorname*{Reg}%
X^{\prime}\right)  \backslash E$ which extends as a diffeomorphism of
manifolds with boundary between some open neighborhoods of $bX$ and
$bX^{\prime}$ in $\overline{X}$ and $\overline{X^{\prime}}$. In particular,
$\varphi\left\vert _{bX}\right.  $ embeds $bX$ into $M$ and $\varphi
^{-1}\left(  \varphi\left(  bX\right)  \right)  =bX$. Note that $\varphi$ may
not preserves node.

An \textit{embedding} is an almost embedding $\varphi:\overline{X}%
\longrightarrow M$ which is an homeomorphism from $\overline{X}$ to
$\varphi\left(  \overline{X}\right)  $, or, equivalently, which is univaluate
and such that for each node $a$ of $X$, the (germs of) branches of $X^{\prime
}$ at $\varphi\left(  a\right)  $ are the images by $\varphi$ of the (germs
of) branches of $X$ at $a$. When $X$ is smooth, these definitions match the
usual ones.

Our first theorem is about boundary data coming from generic almost embeddings
of the interior.

\begin{theorem}
[Generic boundary data]\label{T AEB1}Let $X$ be an open bordered nodal curve,
$c=\left(  c_{\ell}\right)  _{0\leqslant\ell\leqslant2}$ a 3-uple of
admissible families and $G_{c}^{ae}\left(  bX\right)  $ be the set of $\left(
u_{\ell}\right)  _{0\leqslant\ell\leqslant2}\in C^{\infty}\left(  bX\right)
^{3}$ such that $0\notin\left(  \partial U_{0}\right)  \left(  bX\right)  $
and $\left(  \partial U_{0}:\partial U_{1}:\partial U_{2}\right)  $ is an
almost embedding of $\overline{X}$ in $\mathbb{CP}_{2}$ where $U_{\ell}$ is
the unique harmonic distribution extension of $u_{\ell}$ with $c_{\ell}$ as
family of residues. Then $G_{c}^{ae}\left(  bX\right)  $ is a dense open
subset of $C^{\infty}\left(  bX\right)  ^{3}$.
\end{theorem}

Such generic data are closely related to DN-data used in~\cite{HeG-MiV2013} to
solve the inverse Dirichlet-to-Neumann problem~; these DN-data are boundary
data of the form $\left(  \gamma,u,\left(  \partial\widetilde{u}^{c}\right)
\left\vert _{\gamma}\right.  \right)  $ where $\gamma$ is the oriented
boundary of $X$ and $u=\left(  u_{\ell}\right)  _{0\leqslant\ell\leqslant2}\in
G_{c}^{ae}\left(  \gamma\right)  $. Hence, theorem~\ref{T AEB1} tell us in a
certain sense that random data are good. However, it is particularly relevant
for our inverse problem that one can a priori check that a given $u\in
C^{\infty}\left(  bX\right)  ^{3}$ is generic, that is in $G_{c}^{ae}\left(
bX\right)  $. Theorem below, which uses \cite[th. 3a]{HeG-MiV2007} and will be
proved in an other paper, gives a criterion for this genericity question in
terms of shock-wave decomposition of a boundary integral.

\begin{theoremUN}
[Genericity criterion]Let $X$ be an open bordered Riemann surface with
connected boundary $u=\left(  u_{\ell}\right)  _{0\leqslant\ell\leqslant2}\in
C^{\infty}\left(  bX\right)  ^{3}$ such that $0\notin\left(  \partial
\widetilde{u}_{0}\right)  \left(  bX\right)  $ and $\left(  f_{\ell}\right)
_{1\leqslant\ell\leqslant2}=\left(  \frac{\left(  \partial\widetilde{u}%
_{1}\right)  \left\vert _{bX}\right.  }{\left(  \partial\widetilde{u}%
_{0}\right)  \left\vert _{bX}\right.  },\frac{\left(  \partial\widetilde{u}%
_{1}\right)  \left\vert _{bX}\right.  }{\left(  \partial\widetilde{u}%
_{0}\right)  \left\vert _{bX}\right.  }\right)  $ is an embedding. Consider
Cauchy-Fantappie indicatrix of the form $f_{1}d\left(  \xi_{0}+\xi_{1}%
f_{1}+f_{2}\right)  $ which is the function%
\[
G:\mathbb{C}^{2}\ni\left(  \xi_{0},\xi_{1}\right)  \mapsto\frac{1}{2\pi i}%
\int_{\partial X}f_{1}\frac{d\left(  \xi_{0}+\xi_{1}f_{1}+f_{2}\right)  }%
{\xi_{0}+\xi_{1}f_{1}+f_{2}}.
\]
Then $u\in G_{c}^{ae}\left(  bX\right)  $ if and only if there is a non empty
open set $W$ of $\mathbb{C}^{2}$\ and mutually distinct holomorphic functions
$h_{1},...,h_{p}$\ on $W$ satisfying the shock-wave equation $h\frac{\partial
h}{\partial\xi_{0}}=\frac{\partial h}{\partial\xi_{1}}$ such that on $W$%
\begin{equation}
\frac{\partial^{2}}{\partial\xi_{0}^{2}}(G-%
{\displaystyle\sum\limits_{1\leqslant j\leqslant p}}
h_{j})=0~~~\&~~~\frac{\partial^{2}G}{\partial\xi_{0}^{2}}\neq0.
\label{F/ caract G}%
\end{equation}

\end{theoremUN}

When $\left(  u_{\ell}\right)  \in G_{c}^{ae}\left(  bX\right)  $, $Y=\left[
\partial\widetilde{u}_{0}^{c}:\partial\widetilde{u}_{1}^{c}:\partial
\widetilde{u}_{2}^{c}\right]  \left(  X\right)  $ can be seen as a concrete
presentation of $X$. If $Y$ can be explicitly computed by boundary data, one
has a tool to solve the reconstruction problem which is of essential interest
for applications. This is the spirit of the Cauchy type integral formulas
written in~\cite[th. 2]{HeG-MiV2007} and~\cite[th. 5]{HeG-MiV2013}. Note
however that serious effort is still to be made to make these effective~;
\cite[prop. 3.33]{DoP-HeG1997}, \cite{DiT1998a}, \cite[th. 4]{HeG-MiV2007},
\cite[th. 8.3]{HaFLaJ2004} and \cite[th. 1.2]{WaR2008} could be clues for this
goal. The natural question of what to do with non or less generic data is open.

The first step to prove theorem~\ref{T AEB1} is to establish a weak version of
it by an adaptation of the Bishop's technique to produce an affine embedding
of a Stein manifold.

\begin{theorem}
[Weak approximation]\label{P/ PerturbPlgmnt}Let $\Sigma$ be an open nodal
curve and $X\subset\subset\Sigma$ open and smoothly bordered such that
$\Sigma\backslash X\subset\operatorname*{Reg}\Sigma$. For $0\leqslant
\ell\leqslant2$, let $U_{\ell}$ be a harmonic distribution smooth in a
neighborhood of $\Sigma\backslash X$. Then, there is a 3-uple $V=\left(
V_{\ell}\right)  _{0\leqslant\ell\leqslant2}$ of harmonic distributions on
$\Sigma$ smooth near $\Sigma\backslash X$ such that $U-V$ is arbitrarily close
to $0$ in $C^{\infty}\left(  \overline{X}\right)  ^{3}$, $\sigma=\left[
\partial V_{0}:\partial V_{1}:\partial V_{2}\right]  $ is an immersion of
$\overline{X} $ into $\mathbb{CP}_{2}$ which embeds $\gamma=bX$ and $\left(
\sigma\left\vert _{\overline{X}}\right.  \right)  ^{-1}\left(  \delta\right)
=\gamma$ where $\delta=\sigma\left(  \gamma\right)  $. If $U$ is real valued,
$V$ can be chosen so.
\end{theorem}

The second and last step of the proof is theorem below named after the formula
implying that a holomorphic function on an open set of $\mathbb{C}$ continuous
up to the boundary is injective if its boundary restriction is of so.

\begin{theorem}
[Argument principle]\label{P/ Immersion}Let $X$ be an open bordered nodal
curve, $F:\overline{X}\longrightarrow\mathbb{CP}_{2}$ an immersion such that
$F\left\vert _{\gamma}\right.  $ is injective and $F^{-1}\left(  F\left(
bX\right)  \right)  =bX$. Then, $Y=F\left(  X\right)  $ is a complex curve,
$\operatorname*{Sing}Y$ is a finite set and $F$ is an isomorphism from
$X\backslash F^{-1}\left(  \operatorname*{Sing}Y\right)  $ onto
$\operatorname*{Reg}Y$.
\end{theorem}

A slight modification of the proof of theorem~\ref{T AEB1} enables to
establish theorem~\ref{T AEB2} below which is about projective embeddings or
immersions of an open bordered Riemann surface. It can be seen as a variation
of the Bishop-Narasimhan theorem for the first case and of a Bishop
result~\cite{BiE1961} for the second. Theorem~\ref{T AEB2} shows that it is
not necessary to pick up some complicated line bundle to realize a projective
embedding of an open bordered surface since in a certain sense, canonical maps
chosen at random are embeddings. Another interesting feature of
theorem~\ref{T AEB2} is that at the difference of some embedding's theorems,
it doesn't use the genus which, in general, does not come easily at hand,
especially in the inverse problems we are interested in.

\begin{theorem}
[Approximation by almost embeddings]\label{T AEB2}Let $X$ be an open bordered
Riemann surface and $K\left(  \overline{X}\right)  $ the space of holomorphic
$\left(  1,0\right)  $-forms on $X$ which are smooth up to $bX$. If
$n\in\mathbb{N}^{\ast}$, we denotes by $G_{n}^{ae}\left(  X\right)  $ the set
of $\left(  \omega_{\ell}\right)  _{0\leqslant\ell\leqslant n}\in K\left(
\overline{X}\right)  ^{n+1}$ such that $\left[  \omega\right]  $ is an almost
embedding of $\overline{X}$ into $\mathbb{CP}_{n}$ and by $G_{n}^{e}\left(
X\right)  $ the set of those whose associated canonical map is actually an
embedding. Then, $G_{n}^{e}\left(  X\right)  $ is a dense open subset of
$K\left(  \overline{X}\right)  ^{n+1}$ when $n\geqslant3$ and $G_{2}%
^{ae}\left(  X\right)  $ is a dense open subset of $K\left(  \overline
{X}\right)  ^{3}$.
\end{theorem}

This theorem is a particular case of theorem~\ref{T AEB3} for which additional
notation is needed. Let $X$ be a nodal curve. If $c$ is an admissible family,
$K_{c}\left(  \overline{X}\right)  $ is the set of weakly holomorphic $\left(
1,0\right)  $-forms on $X$ which are smooth up to $bX$ near $bX$ and have $c$
as family of residues. $K_{c}\left(  \overline{X}\right)  $ is equipped with
the distance $\left(  \alpha,\beta\right)  \mapsto dist\left(  \alpha
-\beta,0\right)  $ where $dist$ is any distance defining the natural topology
of the canonical bundle $K\left(  \overline{X}\right)  $ of smooth $\left(
1,0\right)  $-forms on $\overline{X}$. If $c=\left(  c_{\ell}\right)
_{0\leqslant\ell\leqslant n}$ a 3-uple of admissible families and write
$c_{\ell}=\left(  c_{p,j}^{\ell}\right)  _{p\in\operatorname*{Sing}%
X,~1\leqslant j\leqslant\nu\left(  p\right)  }$, $0\leqslant\ell\leqslant n$.
If the $c_{p,j}=\left(  c_{p,j}^{\ell}\right)  _{0\leqslant\ell\leqslant n}$
are non zero for any $\left(  p,j\right)  $, $c $ is called a \textit{nodal
family} and a \textit{true nodal family} if in addition for any $p\in
\operatorname*{Sing}X$, the $c_{p,j}$, $1\leqslant j\leqslant\nu\left(
p\right)  $, are on a same complex line of $\mathbb{C}^{3}$. In this case, $c$
is said to be \textit{injective }if the map $\operatorname*{Sing}X\ni
p\mapsto\left[  c_{p,1}^{0}:\cdots:c_{p,1}^{n}\right]  $ is of so.

\begin{theorem}
[Approximation by almost embeddings, nodal case]\label{T AEB3}Let $X$ be an
open bordered nodal curve and $c=\left(  c_{\ell}\right)  _{0\leqslant
\ell\leqslant n}$ an injective nodal family. If $n\in\mathbb{N}^{\ast}$, we
denotes by $G_{n,c}^{e}\left(  \overline{X}\right)  $, resp. $G_{n,c}%
^{ae}\left(  \overline{X}\right)  $, the set of $\omega\in K_{n,c}\left(
\overline{X}\right)  =K_{c_{0}}\left(  \overline{X}\right)  \times\cdots\times
K_{c_{n}}\left(  \overline{X}\right)  $ such that $\left[  \omega\right]  $ is
an embedding, resp. an almost embedding, of $\overline{X} $ into
$\mathbb{CP}_{n}$. Then, $G_{n,c}^{e}\left(  X\right)  $ is a dense open
subset of $K_{n,c}\left(  \overline{X}\right)  $ when $n\geqslant3$ and
$G_{2,c}^{ae}\left(  X\right)  $ is a dense open subset of $K_{2,c}\left(
\overline{X}\right)  $.
\end{theorem}

\noindent That every open bordered nodal curve can be embedded $\mathbb{C}%
^{3}$ is a consequence of a general result of Wiegmann~\cite{WiK1966} and
Sch\"{u}rmann~\cite{ScJ1997}. A nodal version of theorem~A is obtained by
using in its proof theorem~\ref{T AEB3} instead of theorem~\ref{T AEB2}%
.\smallskip

We now turn our attention to the problem of reconstructing a compact Riemann
surface $Z$ from data collected in a small subdomain $S$. More precisely, we
consider the space of functions $U$ which are harmonic outside a finite subset
$P\left(  U\right)  $ of $Z$, have isolated logarithmic singularities at each
point of $P\left(  U\right)  $ and such that $\underset{p\in P\left(
U\right)  }{\Sigma}\operatorname*{Res}\left(  U,p\right)  =0$. This last
condition ensure that $U$ is a harmonic distribution on the nodal surface $X$
where the points of $P\left(  U\right)  $ have been identified. In the sequel,
we speak of such special function as \textit{harmonic distributions}\text{ on
}$Z$.

We denote by $D_{Z}$ the set of $\left(  a,c\right)  $ in $Z^{6}%
\times\mathbb{C}^{3}$ such that $a=\left(  a_{\ell}^{-},a_{\ell}^{+}\right)
_{0\leqslant\ell\leqslant2}$ is a family of six mutually distinct points of
$Z$ and $c=\left(  c_{\ell}\right)  _{0\leqslant\ell\leqslant2}\in
\mathbb{C}^{3}$. If $\left(  a,c\right)  \in D_{Z}$ and $0\leqslant
\ell\leqslant2$, $U_{Z,\ell}^{a,c}$ is a harmonic distribution whose singular
support is $\left\{  a_{\ell}^{-},a_{\ell}^{+}\right\}  $ and has residue $\pm
c_{\ell}$ at $a_{\ell}^{\pm}$~; as a matter of fact, $U_{Z,\ell}^{a,c}$ is a
standard Green bipolar function and while it is determined only up to an
additive constant, $\partial U_{Z,\ell}^{a,c}$ is unique.

For $n\in\mathbb{N}^{\ast}$, $D_{Z,n}$ is the set of $\left(  a,c,p,\kappa
\right)  $ in $Z^{6}\times\mathbb{C}^{3}\times\left(  Z^{n}\right)  ^{3}%
\times\left(  \mathbb{C}^{n}\right)  ^{3}$ such that $\left(  a,c\right)  \in
D_{Z}$ and for any $\ell\in\left\{  0,1,2\right\}  $, $p_{\ell}=\left(
p_{\ell,j}\right)  _{1\leqslant j\leqslant n}$ is a family of mutually
distinct points of $Z\backslash\left\{  a_{0}^{-},a_{0}^{+},a_{1}^{-}%
,a_{1}^{+},a_{2}^{-},a_{2}^{+}\right\}  $ and $\kappa_{\ell}=\left(
\kappa_{\ell,j}\right)  _{1\leqslant j\leqslant n}\in\mathbb{C}^{n}$ satisfies
$%
{\displaystyle\sum\limits_{1\leqslant j\leqslant n}}
\kappa_{\ell,j}=0$.

If $\left(  a,c,p,\kappa\right)  \in D_{Z,n}$, we denote by $V_{Z,\ell
}^{p,\kappa}$ a harmonic distribution with $\left\{  p_{\ell,j};~1\leqslant
j\leqslant n\right\}  $ as singular support and residue $\kappa_{\ell,j}$ at
$p_{\ell,j}$, $1\leqslant j\leqslant n$~;$\ V_{Z}^{p,\kappa}$ is unique up to
an additive constant~; we set $U_{Z,\ell}^{a,c,p,\kappa}=U_{Z,\ell}%
^{a,c}+V_{Z,\ell}^{p,\kappa}$, $0\leqslant\ell\leqslant2$. We denote by
$E_{Z,n}\left(  S\right)  $ the set of $\left(  a,c,p,\kappa\right)  \in
D_{Z,n}$ such that
\[
F_{Z}^{a,c,p,\kappa}=\left[  \partial U_{Z,0}^{a,c,p,\kappa}:\partial
U_{Z,1}^{a,c,p,\kappa}:\partial U_{Z,2}^{a,c,p,\kappa}\right]
\]
is well defined and injective outside some finite subset of $Z\backslash S$ ;
$E_{Z,n}\left(  S\right)  $ is open in $D_{Z,n}$.

Note that since $F_{Z}^{a,c,p,\kappa}$ takes the value $\left(  1,0,0\right)
$ at each pole of $\partial U_{Z,0}^{a,c,p,\kappa}$, the map $F_{Z}%
^{a,c,p,\kappa}$ can't be an embedding should it be well defined on the whole
of $Z$. It can't be neither an almost embedding since its degree would be at
least $2$. Hence, the result below which establishes the genericity assumption
claimed in~\cite[th.~1]{HeG-MiV2013} is somehow delicate. In this theorem, $S$
should be considered as a known subdomain where essentially the non
injectivity of the maps are confined and $\left(  p,\kappa\right)  $.

\begin{theorem}
[Almost embedding by dipole perturbation]\label{T/ compact generique}Let
$Z$\ be a compact Riemann surface and $S$ a non empty open subset of $Z$.
Consider $\left(  a,c\right)  $ in $D_{Z}$ with $a\in S^{6}$. Then for any
$\varepsilon\in\mathbb{R}_{+}^{\ast}$, there exists $n\in\mathbb{N}^{\ast}$
and $\left(  p,\kappa\right)  \in\left(  S^{n}\right)  ^{3}\times\left(
\mathbb{C}^{n}\right)  ^{3}$ such that $\left(  a,c,p,\kappa\right)  \in
E_{Z,n}\left(  S\right)  $ and $\left\vert \kappa\right\vert _{1}%
\overset{def}{=}%
{\displaystyle\sum\limits_{0\leqslant\ell\leqslant2,\ 1\leqslant j\leqslant
n}}
\left\vert \kappa_{\ell,j}\right\vert \leqslant\varepsilon$.
\end{theorem}

This theorem is proved by applying theorem~\ref{T AEB1} to $Z\backslash S$ and
theorem~\ref{T/ Runge} below . Note that this results and others of
\cite{BoAGaP1984, BaT1988} can been seen as a development or improvement of a
theorem of Runge-Behnke-Stein~\cite{BeHStK1949, FoOL1977} for holomorphic
functions on an open Riemann surface which is quoted by Remmert in his
book~\cite{ReR1998L} as follows\medskip

\noindent\textbf{Theorem (}Behnke-Stein). \textit{For every subdomain }$D$
\textit{of a noncompact Riemann surface }$Z$\textit{, there exists a set }$T$
\textit{of boundary points of }$D$ \textit{(in }$Z$\textit{) that it is at
most countable and has the following property: Every function in }%
$\mathcal{O}\left(  D\right)  $ \textit{can be approximated compactly in }$D$
\textit{by functions meromorphic in }$Z$\textit{\ that each have finitely many
poles, all which lie in}~$T$\textit{.}\smallskip

In the statement below, $L_{m}^{1}\left(  Z\backslash S\right)  $,
$m\in\mathbb{N}$, is the Sobolev space of distributions on $Z\backslash
\overline{S}$ whose total differentials up to order $m$ are integrable on
$Z\backslash S$, $Z$ being equipped with any hermitian metric.

\begin{theorem}
[Quantitative Runge-harmonic approximation]\label{T/ Runge}Let $Z$\ be a
compact connected oriented smooth Riemann surface and $S$ a smoothly bordered
open subset of $Z$. Then, there exists $\left(  C_{m}\right)  \in
\mathbb{R}_{+}^{\mathbb{N}}$ depending only of $S$ such that for every
$\varepsilon\in\mathbb{R}_{+}^{\ast}$ and $\varphi\in C^{\infty}\left(
Z\backslash S\right)  $ harmonic in $Z\backslash\overline{S}$, there is a
finite subset $P_{\varepsilon}$ of $S$ and a function $\varphi_{\varepsilon}$
harmonic in $Z\backslash P_{\varepsilon}$ with isolated logarithmic
singularities such that $\left\Vert \varphi-\varphi_{\varepsilon}\right\Vert
_{C^{m}\left(  Z\backslash S\right)  }\leqslant C_{m}\varepsilon~\left\Vert
\varphi\right\Vert _{_{L_{2}^{1}\left(  Z\backslash S\right)  }}$ for any
$m\in\mathbb{N}$ and $%
{\displaystyle\sum\limits_{a\in P_{\varepsilon}}}
\left\vert \operatorname*{Res}\left(  \varphi_{\varepsilon},a\right)
\right\vert \leqslant C_{0}\left\Vert \varphi\right\Vert _{_{L_{2}^{1}\left(
Z\backslash S\right)  }}$. If $\varphi$ is real valued, $\varphi_{\varepsilon
}$ can be chosen real valued.
\end{theorem}

\section{Proofs for the open bordered case\label{S 3}}

We first establish theorem~\ref{P/ Immersion}.

\begin{proof}
[\textbf{Proof of theorem~\ref{P/ Immersion}}]As an open bordered nodal curve
is an open bordered Riemann surface where a finite number of points have been
identified, it is sufficient to prove this proposition when $X$ is smooth.
Since $F^{-1}\left(  \delta\right)  =\gamma$, $\left.  F\right\vert _{X}$ is
proper and we know by a theorem of Remmert~\cite{ReR1956} that $Y$\ is an
analytic subset of $\mathbb{CP}_{2}\backslash\delta$~; $Y$ has pure dimension
$1$ because $F$ is an immersion. By reasoning with a connected component of
$X$, we reduce the proof to the case where $X$ is connected. Then, $Y$ is
connected and because $F$ is an immersion, $F$ has a degree $\nu$ over
$\operatorname*{Reg}Y$ which is defined by%
\[
\nu=\max\left\{  \operatorname{Card}\left(  X\cap F^{-1}\left(  \left\{
y\right\}  \right)  \right)  ~;~y\in\operatorname*{Reg}Y\right\}
\]
As $F\left\vert _{X}\right.  $ is proper, $F_{\ast}\left[  X\right]  $ is a
well defined and equals $\nu\left[  Y\right]  $. Hence $\nu d\left[  Y\right]
=dF_{\ast}\left[  X\right]  =F_{\ast}d\left[  X\right]  =F_{\ast}\left[
\gamma\right]  =\left[  \delta\right]  $ because $F\left\vert _{\gamma
}\right.  $ is injective. As $F$ is an immersion, $d\left[  Y\right]  $ is
locally of the form $\pm\left[  \delta\right]  $ and $\nu=1$.
\end{proof}

Lemma~\ref{L/ proj=can} below shows that theorem~A is a consequence of
theorem~\ref{T AEB2}.

\begin{lemma}
\label{L/ proj=can}Every holomorphic projective map of an open or open
bordered Riemann surface is canonical.
\end{lemma}

\begin{proof}
Let $R$ be a Riemann surface as above and $F:\overline{R}\longrightarrow
\mathbb{CP}_{n}$ a holomorphic map, smooth up to the boundary if $R$ is
bordered. We assume that $R$ is connected without loss of generality. Let
$\left(  t_{0},...,t_{n}\right)  $ be the homogeneous coordinates of
$\mathbb{CP}_{n}$ and for $0\leqslant j\leqslant n$, $T_{j}=\left\{  t_{j}%
\neq0\right\}  $ and $\zeta_{j}=\left(  \zeta_{j,k}\right)  _{k\neq j}=\left(
t_{j}/t_{k}\right)  $ the natural affine coordinates for $\mathbb{CP}_{n}$ in
$T_{j}$~; we set $\zeta_{j,j}=1$. Then the functions $f_{j,k}=\zeta_{j,k}\circ
F$ are a data for a multiplicative Cousin problem on $Z$ associated to the
covering $\left(  R_{j}\right)  _{0\leqslant j\leqslant n}=\left(
F^{-1}\left(  T_{j}\right)  \right)  _{0\leqslant j\leqslant n}$. Original
proofs of that such a problem has always a solution on an open or bordered
Riemann surface can be found in the paper of Behnke-Stein~\cite{BeHStK1949}
for the first case and can be deduced from the results of Koppelman in
\cite{KoW1959} about $\overline{\partial}$-resolution with regularity up to
the boundary. So, we can find $\left(  f_{j}\right)  \in\mathcal{O}^{\ast
}\left(  R_{0}\right)  \times\cdots\times\mathcal{O}^{\ast}\left(
R_{n}\right)  $ such that $f_{j,k}=f_{j}/f_{k}$ on each $R_{j}\cap R_{k}$.
When $k$ is fixed, the relations $f_{k}=f_{j,k}f_{j}$ shows that $f_{k}$
extends holomorphically to $R$ and smoothly to $\overline{R}$. Hence, the
$f_{j,k}$ extends meromorphically on $R$ and the relations $f_{j,k}%
=f_{k}/f_{j}$ hold on $\overline{R}$. As a map on $\overline{R}$, $f=\left(
f_{k}\right)  $ have no zero and $\left[  f\right]  $ is well defined.
$F=\left[  f\right]  $ because on each $R_{j}$, $[\left(  f_{k}/f_{j}\right)
_{0\leqslant k\leqslant n}]=[\left(  f_{j,k}\right)  _{0\leqslant k\leqslant
n}]$.

Consider now $\alpha_{0}\in K\left(  \overline{R}\right)  $ not identically
zero~; if $R$ is bordered, we choose for $\alpha_{0}$ the restriction to
$\overline{R}$ of a non zero element of $K\left(  R^{\prime}\right)  $ where
$R^{\prime}$ is some open connected neighborhood of $\overline{R}$ in its
double. Then, using the Weierstrass theorem, we take $A\in\mathcal{O}\left(
R\right)  \cap C^{\infty}\left(  \overline{R}\right)  $ whose divisor is the
divisor of $\alpha_{0}$. The continuous extension $\alpha$ of $\frac{1}%
{A}\alpha_{0}$ is in $K\left(  \overline{R}\right)  $ and never vanishes.
Hence, $F=\left[  f\alpha\right]  $ is a canonical map.
\end{proof}

Subsequent statements and proofs are technically complicated by nodes'
existence but one can easily isolate the smooth case. Theorem~\ref{T AEB2} is
a particular case of theorem~\ref{T AEB3} and if along the lines below
establishing theorem~\ref{T AEB1}, one forgets that the $\left(  1,0\right)
$-forms appearing are of the image by $\partial$ of a harmonic distribution,
one gets a proof for theorem~\ref{T AEB3}, strictly speaking when $n=3$ but
the case $n\geqslant4$ is obtained the same way. The next two lemmas reduce
theorem's~\ref{T AEB1} proof to the case where hypothesis of
theorem~\ref{P/ PerturbPlgmnt} hold. Theorems~\ref{P/ PerturbPlgmnt} and
\ref{P/ Immersion} gives theorem~\ref{T AEB1} in the reduced case and hence in
general. We set now notations for the remainder of the section.

Let $X$ be an open bordered nodal curve. Since $X$ can be seen as a subdomain
of its double, we assume that $X$ is a relatively compact open subset of an
open nodal curve $\Sigma$ whose all components meet $X$ and such that
$\Sigma\backslash X\subset\operatorname*{Reg}\Sigma$ and $X\cap
\operatorname*{Reg}\Sigma$ is smoothly bordered in $\operatorname*{Reg}\Sigma
$. In the sequel $\Sigma$ is replaced by a sufficiently small open
neighborhood of $\overline{X}$ when needed.

Let $R$ be the smooth Riemann surface whose $\Sigma$ is a quotient and
$\pi:R\longrightarrow\Sigma$ the natural projection. We set $W=\pi^{-1}\left(
X\right)  $ and
\[
S=\pi^{-1}\left(  \operatorname*{Sing}\Sigma\right)  =\left\{  p_{j}%
~;~p\in\operatorname*{Sing}X~\&~1\leqslant j\leqslant\nu\left(  p\right)
\right\}
\]
where $\left\{  p_{1},...,p_{\nu\left(  p\right)  }\right\}  =\pi^{-1}\left(
p\right)  $ when $p\in\operatorname*{Sing}X$. As $\gamma\subset
\operatorname*{Reg}\Sigma$, we identify $\pi^{-1}\left(  \gamma\right)  $ and
$\gamma$.

Note that the pullback by $\pi$ of a harmonic distribution (resp. a weakly
holomorphic $\left(  1,0\right)  $-form) on $X$ is a harmonic function (resp.
a holomorphic $\left(  1,0\right)  $-form) on $W\backslash S$ with at most
isolated logarithmic singularities (resp. simple poles) and such that for any
$p\in\operatorname*{Sing}X$, the sum of its residues at $p$ vanishes.
Conversely such harmonic functions or holomorphic $\left(  1,0\right)  $-forms
on $W\backslash S$ have a well defined direct image by $\pi$ as a harmonic
distribution or weakly holomorphic $\left(  1,0\right)  $-form)~.

We fix a 3-uple $c=\left(  c_{\ell}\right)  _{0\leqslant\ell\leqslant2}$ of
admissible families and write $c_{\ell}=\left(  c_{p,j}^{\ell}\right)
_{p\in\operatorname*{Sing}X,~1\leqslant j\leqslant\nu\left(  p\right)  }$. In
the sequel $u=\left(  u_{\ell}\right)  _{0\leqslant\ell\leqslant2}$ where for
$0\leqslant\ell\leqslant2$, $u_{\ell}$ is the restriction to $bX$ of a
harmonic distribution $U_{\ell}$ on $\Sigma$ smooth near $\Sigma\backslash X$
with $c_{\ell}$ as family of residues~; we set $\omega=\left(  \omega_{\ell
}\right)  _{0\leqslant\ell\leqslant2}=\left(  \partial U_{\ell}\right)
_{0\leqslant\ell\leqslant2}$, $\theta=\left(  \theta_{\ell}\right)
_{0\leqslant\ell\leqslant2}=\left(  \pi^{\ast}\omega_{\ell}\right)
_{0\leqslant\ell\leqslant2}$ and for $0\leqslant\ell\leqslant2$,%
\[
S_{\ell}=\left\{  p\in S~;~\operatorname*{Res}\left(  \theta_{\ell},p\right)
\neq0\right\}  ~\text{and}~Z_{\ell}=\left\{  \theta_{\ell}=0\right\}  .
\]

If $p\in\operatorname*{Sing}X$ and $1\leqslant j\leqslant\nu\left(  p\right)
$, $p_{j}\in Z_{\ell}$ if and only if $c_{p,j}^{\ell}=0$ and $\theta_{\ell}$
is continuously prolonged by $0$ at $p$. When this occurs for any $\ell
\in\left\{  0,1,2\right\}  $, our data corresponds to a nodal curve where only
the points of $\pi^{-1}\left(  p\right)  \backslash\left\{  p_{j}\right\}  $
have been identified. Hence, we can assume from now for this section that the
following holds

\begin{center}
\textit{either }$c=0$, \textit{either }$c$\textit{\ is a nodal family. }
\end{center}

\noindent The first case means that we are working on the smooth Riemann
surfaces $W$ and $R$ with usual harmonic functions and holomorphic forms. In
the second case, $\left[  \theta\right]  $ is well defined on each $p_{j}$
where it takes the value $\left[  c_{p,j}^{0}:c_{p,j}^{1}:c_{p,j}^{2}\right]
$. When $c$ is a true nodal family, $\left[  \omega\right]  $ is univaluate at
the nodes of $X$ and can't be an embedding is $c$ is not injective.

Our approximation process needs the following lemma whose statement firstly
appears in~\cite{BiE1958} for the holomorphic case. Our proof uses
$\overline{\partial}$ and $\partial\overline{\partial}$ equation.

\begin{lemma}
\label{L/ RungefaibleLisse}Let $Z$ be a smoothly bordered relatively compact
open subset of an open Riemann surface $R$. Then, there exists an open
neighborhood $Z^{\prime}$ of $\overline{Z}$ in $R$ such that $\left(
Z,Z^{\prime}\right)  $ is a harmonic pair (resp. a Runge pair) which means
that any element of $C^{\infty}\left(  \overline{Z}\right)  $ (resp.
$C_{1,0}^{\infty}\left(  \overline{Z}\right)  $) harmonic (resp. holomorphic)
on $Z$ can be arbitrarily approximated in $C^{\infty}\left(  \overline
{Z}\right)  $ (resp. $C_{1,0}^{\infty}\left(  \overline{Z}\right)  $) by
harmonic functions (resp. holomorphic $\left(  1,0\right)  $-forms) on
$Z^{\prime}$.
\end{lemma}

\begin{proof}
First consider two harmonic pairs $\left(  A_{1},A_{1}^{\prime}\right)  $ and
$\left(  A_{2},A_{2}^{\prime}\right)  $. If $I^{\prime}=A_{1}^{\prime}\cap
A_{2}^{\prime}$ is empty, $\left(  A,A^{\prime}\right)  =\left(  A_{1}\cup
A_{2},A_{1}^{\prime}\cup A_{2}\right)  $ is a harmonic pair. Suppose
$I^{\prime}\neq\varnothing$, $U\in C^{\infty}\left(  \overline{C}\right)  $ is
harmonic in $C$, $U_{j}$ ($j=1,2$) is harmonic on $A_{j}^{\prime}$ and set
$V=U_{1}-U_{2}$ on $I^{\prime}$. Fix $\chi_{1}\in C^{\infty}\left(  A^{\prime
},\left[  0,1\right]  \right)  $ with support in $A_{1}^{\prime}$ such that
$\chi_{1}=1$ near $A_{1}^{\prime}\backslash A_{2}^{\prime}$ in $A^{\prime}$
and for $\left\{  j,k\right\}  =\left\{  1,2\right\}  $, let $V_{k}$ be the
null extension of $\left.  \left(  \chi_{j}V\right)  \right\vert _{I^{\prime}%
}$ to $A_{k}^{\prime}\backslash I^{\prime}$ where $\chi_{2}=1-\chi_{1}$. These
functions are smooth and because $V_{1}+V_{2}=V$ is harmonic on $I^{\prime}$,
we can define $\Phi\in C_{1,1}^{\infty}\left(  A^{\prime}\right)  $ by the
formulas $\left.  \Phi\right\vert _{A_{k}^{\prime}}=i\left(  -1\right)
^{k}\partial\overline{\partial}V_{j}$, $\left\{  j,k\right\}  =\left\{
1,2\right\}  $. According to e.g. the original work of
Koppelman~\cite{KoW1959} on $\overline{\partial}$ for Riemann surfaces, the
equation $\Phi=i\partial\overline{\partial}\varphi$ can be solved with
$\varphi\in C^{\infty}\left(  A^{\prime}\right)  $ such that for any integer
$m$%
\begin{equation}
\left\Vert \varphi\right\Vert _{C^{m}\left(  \overline{Z}\right)  }\leqslant
C_{m}\left\Vert \Phi\right\Vert _{C^{m}\left(  \overline{Z}\right)  }
\label{F/ controleL}%
\end{equation}
where $C_{m}$ depends only of $m$ and $A^{\prime}$. The function $U^{\prime}$
well defined on $A^{\prime}$ by $\left.  U^{\prime}\right\vert _{A_{j}%
^{\prime}}=U_{j}-\left(  -1\right)  ^{j}V_{j}-\left.  \varphi\right\vert
_{A_{j}^{\prime}}$, $j=1,2$, is harmonic. By construction $U^{\prime}-U$
equals $\chi_{1}\left(  U_{1}-U\right)  +\chi_{2}\left(  U_{2}-U\right)
-\varphi$ on $I^{\prime}$ and $U_{j}-U-\varphi$ on $A_{j}\backslash I^{\prime
}$, $j=1,2$. Thus, if $m$ is any integer,
\[
\left\Vert U^{\prime}-U\right\Vert _{C^{m}\left(  A\right)  }\leqslant
C_{m}^{\prime}\left(  \left\Vert U_{1}-U\right\Vert _{C^{m}\left(
A_{1}\right)  }+\left\Vert U_{2}-U\right\Vert _{C^{m}\left(  A_{2}\right)
}+\left\Vert \varphi\right\Vert _{C^{m}\left(  A\right)  }\right)
\]
where $C_{m}^{\prime}$ depends only of $\chi_{1}$. With~(\ref{F/ controleL}),
we get that $U^{\prime}$ is close to $U$ in $C^{\infty}\left(  \overline
{A}\right)  $ as wished provided that $\left(  U_{1},U_{2}\right)  $ is close
enough to $\left(  U,U\right)  $ in $C^{\infty}\left(  \overline{A_{1}}%
\times\overline{A_{2}}\right)  $. This proves that $\left(  A,A^{\prime
}\right)  $ is a harmonic pair.

Since $Z$ is smoothly bordered, $Z$ is covered by a finite family $\left(
D_{j}\right)  _{1\leqslant j\leqslant n}$ of open subsets of $R$ such that for
any $j$, $\overline{D_{j}}$ is contained in a coordinate patch of $R$ and
$Z_{j}=Z\cap D_{j}$ is simply connected. For each $j$, we then select an open
neighborhood $D_{j}^{\prime}$ of $\overline{Z_{j}}$ in $R$ and we set
$Z^{\prime}=\underset{1\leqslant j\leqslant n}{\cup}D_{j}^{\prime}$. Classical
Runge theorem (1885, see also \cite{BaT1988, BoAGaP1984, ReR1998L}) in
$\mathbb{C}$ implies that $\left(  D_{j},D_{j}^{\prime}\right)  $ are harmonic
pairs. Applying inductively what precedes to the pairs $(\underset{1\leqslant
j\leqslant k}{\cup}D_{j},\underset{1\leqslant j\leqslant k}{\cup}D_{j}%
^{\prime})$, we get that $\left(  Z,Z^{\prime}\right)  $ is a harmonic pair.

The same technique but with $\overline{\partial}$-resolution of $\left(
1,1\right)  $-forms on open Riemann surface gives the statement for $\left(
1,0\right)  $-forms.
\end{proof}

When $A$ and $A^{\prime}$ are open subsets of $\Sigma$ such that
$A\subset\subset A^{\prime}$, $\left(  A,A^{\prime}\right)  $ is called a
harmonic pair if for any harmonic distribution $U$ on $A$ smooth near $bA$ in
$\overline{A}$, there is a harmonic distribution $V$ on $A^{\prime}$ whose
singularities and residues on $A^{\prime}$ are those of $U$ on $A$ such that
$V-U$ is arbitrarily close to $0$ in $C^{\infty}\left(  \overline{A}\right)
$. Runge pairs are defined likewise. Our approximation process starts with the
following corollary of lemma~\ref{L/ RungefaibleLisse}.

\begin{lemma}
\label{L/ RungefaibleNodal}Let $\Sigma$ be an open nodal curve and $X$ a
relatively compact open smoothly bordered subset of $\Sigma$ such that
$\left(  bX\right)  \cap\operatorname{Sing}\Sigma=\varnothing$. Then, there
exists an open neighborhood $X^{\prime}$ of $\overline{X}$ in $\Sigma$ such
that $\left(  X,X^{\prime}\right)  $ is a harmonic pair and a Runge pair.
\end{lemma}

\begin{proof}
Let $\Sigma^{\prime}$ and $Z$ smoothly relatively compact open neighborhoods
of $bX$ in $\operatorname*{Reg}\Sigma$ such that $Z\subset\subset
\Sigma^{\prime}$. Applying lemma~\ref{L/ RungefaibleLisse} to $Z\cap X$ and
$\Sigma^{\prime}$, we get an open neighborhood $Z^{\prime}$ of $Z$ in
$\Sigma^{\prime}$ such that $\left(  Z,Z^{\prime}\right)  $ is a harmonic and
Runge pair. We fix a neighborhood $X^{\prime}$ of $\overline{X}$ in
$X\cup\Sigma^{\prime}$ such that $X^{\prime}\cap\Sigma^{\prime}$ is a smooth
open subset of $\Sigma^{\prime}$.

Let $U$ be a harmonic distribution on $X$ smooth near $bX$ in $\overline{X}$,
$m$ an integer and $\varepsilon\in\mathbb{R}_{+}^{\ast}$. As $U$ is a smooth
function on $\overline{Z}$, we can find $H$ harmonic on $Z^{\prime}$ such that
$\left\Vert H-U\right\Vert _{C^{m}\left(  \overline{Z}\right)  }%
\leqslant\varepsilon$. According to e.g. \cite[prop.~2]{HeG-MiV2013}, there is
a unique harmonic distribution $V$ on $X^{\prime}$ whose restriction $v$ to
$bX^{\prime}$ is $w=H\left\vert _{bX^{\prime}}\right.  $ and which has the
same singularities and residues as $U$. Then $\pi^{\ast}\left(  V-U\right)  $
extends has a harmonic function on $W$ smooth on $\overline{W}$. So, the
maximum principle gives $\left\Vert H-U\right\Vert _{C^{m}\left(  \overline
{X}\right)  }\leqslant\varepsilon$.

Let now $\alpha$ be a (weakly) holomorphic (1,0)-form on $X$ smooth up to $bX$
near $bX$. We apply the same technique as in lemma's~\ref{L/ RungefaibleLisse}
proof with $A_{1}=Z$, $A_{1}^{\prime}=Z^{\prime}$, $A_{2} $ and $A_{2}%
^{\prime}$ open subsets of $X$ with complements in $X$ contained in $Z$ and
with smooth boundaries. We choose an approximation $\alpha_{1}^{\prime}\in
C_{1,0}^{\infty}\left(  Z^{\prime}\right)  $ of $\alpha_{1}=\alpha\left\vert
_{Z}\right.  $, $\chi\in C_{c}^{\infty}\left(  Z^{\prime}\right)  $ such that
$\chi_{1}=1$ on a neighborhood of $Z^{\prime}\backslash A_{2}^{\prime}$ in
$Z^{\prime}$ and consider the natural extension $\Phi=\left(  \overline
{\partial}\chi\right)  \wedge\left(  \alpha_{1}-\alpha\right)  $. Select
$\psi\in C_{1,0}^{\infty}\left(  R\right)  $ such that $\pi_{\ast}%
\Phi=\overline{\partial}\psi$ and $\left\Vert \psi\right\Vert _{C^{m}\left(
\pi^{-1}\left(  X\right)  \right)  }\leqslant C_{m}\left\Vert \Phi\right\Vert
_{C^{m}\left(  X\right)  }$. As $\operatorname*{Supp}\Phi\subset
\operatorname*{Reg}\Sigma$, $\varphi=\pi_{\ast}\psi$ is smooth on every branch
of $X$ and satisfies $\Phi=i\overline{\partial}\varphi$ on $\Sigma$. Hence the
form defined by $\left.  \alpha^{\prime}\right\vert _{A_{j}^{\prime}}%
=\alpha_{j}-\left(  -1\right)  ^{j}\alpha_{j}-\left.  \varphi\right\vert
_{A_{j}^{\prime}}$, $j=1,2$ has the same singularities and residues as
$\alpha$ and is arbitrarily close to $\alpha$ in $C_{1,0}^{\infty}\left(
\overline{X}\right)  $ provided $\alpha_{1}^{\prime}$ is arbitrarily close of
$\alpha_{1}$ in $C_{1,0}^{\infty}\left(  \overline{Z}\right)  $.
\end{proof}

Lemma~\ref{L/ RungefaibleNodal} gives a 3-uple $\left(  V_{\ell}\right)
_{0\leqslant\ell\leqslant2}$ of harmonic distributions near $\overline{X}$
such that $\left(  V_{\ell}-U_{\ell}\right)  _{0\leqslant\ell\leqslant2}$ is
arbitrarily small in $C^{\infty}\left(  \overline{X}\right)  ^{3}$. Thus, it
is sufficient to prove theorem~\ref{T AEB1} when $U$ is a 3-uple of harmonic
distributions on $\Sigma$ smooth near $\Sigma\backslash X$~; a likewise
reduction holds for theorem~\ref{T AEB3}. We have now to prove
theorem~\ref{P/ PerturbPlgmnt}. This done with lemmas~\ref{L/ 1ereAprox}
to~\ref{L/ plgmTot} which mainly rely on the fact that as $R$ is a Stein
manifold, Oka-Cartan techniques or $\overline{\partial}$-techniques\ involving
$L^{2}$-estimates with singular plurisubharmonic weights enables to prove the
following~:\smallskip

(1) If $z\in R$, for any holomorphic $\left(  1,0\right)  $-form $\Omega$ non
vanishing at $z$, there is $h\in\mathcal{O}\left(  R\right)  $ such that
$\partial_{\Omega}h\overset{def}{=}\frac{\partial h}{\Omega}$ is a coordinate
for $R$ near $z$.

2) If $z,z^{\prime}\in R$ and $z\neq z^{\prime}$, for any holomorphic $\left(
1,0\right)  $-form $\Omega$ non vanishing at $z$ and $z^{\prime}$, there is
$h\in\mathcal{O}\left(  R\right)  $ such that $\left(  \partial_{\Omega
}h\right)  \left(  z\right)  \neq\left(  \partial_{\Omega}h\right)  \left(
z^{\prime}\right)  $.\smallskip

We first use this tool in lemma~\ref{L/ 1ereAprox} to approximate $U_{0}$ and
$U_{1}$ by harmonic distributions on $\Sigma$ whose derivative induces a
canonical map.

\begin{lemma}
\label{L/ 1ereAprox}There exists harmonic distributions $U_{0}^{\prime}$ and
$U_{1}^{\prime}$ on $\Sigma$ such that for $U^{\prime}=\left(  U_{0}^{\prime
},U_{1}^{\prime},U_{2}\right)  $, $U^{\prime}-U$ is arbitrarily close to $0$
in $C^{\infty}\left(  \overline{X}\right)  ^{2}$, $0\notin\left(  \partial
U_{0}^{\prime}\right)  \left(  \gamma\right)  $ and $\left(  \partial
U_{0}^{\prime},\partial U_{1}^{\prime}\right)  $ (resp. $\partial U^{\prime}$)
induces a well defined canonical map on a neighborhood of $\operatorname*{Reg}%
\overline{X}$ (resp. $\overline{X}$) in $\Sigma$. Moreover, $U$ can also be
chosen real valued if $u$ is real valued.
\end{lemma}

\begin{proof}
As in lemma's~\ref{L/ proj=can} proof, we use an auxiliary form $\alpha\in
K\left(  R\right)  $ nowhere vanishing. Since $\overline{W}$ is compact, (2)
enables to find $\eta_{1},...,\eta_{m}\in\mathcal{O}\left(  R\right)  $ and a
covering $A_{1},...A_{m}$ of $\overline{W}$ by open subsets of $R$ such that
$0\notin\left(  \partial_{\alpha}\eta_{j}\right)  \left(  \overline{A_{j}%
}\right)  $, $1\leqslant j\leqslant m$. The Whiney-Sard lemma implies that
$\left(  \frac{\theta_{0}/\alpha}{\partial_{\alpha}\eta_{1}},\frac{\theta
_{1}/\alpha}{\partial_{\alpha}\eta_{1}}\right)  \left(  \overline{A_{1}%
}\backslash S_{\ell}\right)  $ and $\frac{\theta_{0}/\alpha}{\partial_{\alpha
}\eta_{1}}\left(  \gamma\cap A_{1}\right)  $ (which is a subset of
$\mathbb{C}$ since $S_{0}\cap\gamma=\varnothing$) have Lebesgue measure $0$.
Hence, for almost all $\varepsilon_{1}$ in $\mathbb{C}^{2}$, $\left(
\theta_{0}^{1},\theta_{1}^{1}\right)  =\left(  \theta_{\ell}-\varepsilon
_{1,\ell}\partial\eta_{1}\right)  _{0\leqslant\ell\leqslant1}$ and $\theta
_{0}^{1}$ never vanish on $\overline{A_{1}}$ and $\gamma\cap\overline{A_{1}}$
respectively. If $S\neq\varnothing$ and $p\in S$, $\theta_{\ell}^{1}$ and
$\theta_{\ell}$, $0\leqslant\ell\leqslant1$, have the same residue at $p$ and
as $c$ is nodal, $\left(  \theta_{0}^{1}:\theta_{1}^{1}:\theta_{2}^{1}\right)
$ is well defined at $p$.

Iterating this argument we get that for arbitrarily small $\varepsilon
_{1},...,\varepsilon_{m}$ in $\mathbb{C}^{2}$ with the property that
$\varepsilon_{j+1}$ is small enough with respect to $\varepsilon_{j}$,
$\left(  \theta_{0}^{m},\theta_{1}^{m}\right)  =\left(  \theta_{\ell
}+\varepsilon_{1,\ell}\partial\eta_{1}+\cdots+\varepsilon_{m,\ell}\partial
\eta_{m}\right)  $ induces a well defined canonical map near $\overline
{W}\backslash S$ and that $0\notin\theta_{0}^{m}\left(  \gamma\right)  $. As
$\theta_{\ell}^{m}=\partial\left(  U_{\ell}+\operatorname{Re}B_{\ell}\right)
$ where $B_{\ell}=\underset{1\leqslant j\leqslant r}{\Sigma}\varepsilon
_{j,\ell}\eta_{j}\in\mathcal{O}\left(  R\right)  $, $B\left\vert
_{\overline{W}}\right.  $ is arbitrarily small in $C^{\infty}\left(
\overline{W}\right)  $ and $\left(  U_{\ell}^{\prime}\right)  =\left(
U_{\ell}+\pi_{\ast}\operatorname{Re}B_{\ell}\right)  $ fulfills our demand.
\end{proof}

Applying lemma~\ref{L/ 1ereAprox} to $U=0$ in the smooth case and choosing for
$R$ a sufficiently small open neighborhood of $\overline{W}$, we get $\left(
\eta_{0},\eta_{1}\right)  \in\mathcal{O}\left(  R\right)  ^{2}$ such that
$\left(  \alpha_{0},\alpha_{1}\right)  =\left(  \partial\eta_{0},\partial
\eta_{1}\right)  $ never vanishes. This will be useful to keep the nature of
approximating forms in what follows.

In the general case, lemma~\ref{L/ 1ereAprox} enables to assume that $\left[
\theta\right]  $ is well defined and that $0\notin\theta_{0}\left(
\gamma\right)  $. Thus $\gamma\subset W_{0}=\overline{W}\backslash Z_{0}$ and
$\overline{W}$ is the disjoint union of $W_{0}$ and of the finite sets
$W_{1}=\left(  \overline{W}\cap Z_{0}\right)  \backslash Z_{1}$ and
$W_{2}=\overline{W}\cap Z_{0}\cap Z_{1}\subset\overline{W}\backslash Z_{2}$.
Note that at points of $W_{\ell}$, $\theta_{\ell}$ has a simple pole or is non zero.

The next lemmas follow the path of the Bishop's proof that a Stein manifold of
dimension $d$ can be embedded in $\mathbb{C}^{2d+1}$. Actually, as our goal is
much simpler than the results of Bishop~\cite{BiE1961}, we use the simplified
lecture of~\cite{HorL1965L}. Starting with $\omega$ we use (1) and (2) to find
sufficiently many $\theta_{j}$ so that $\left(  \theta_{0}:\cdots:\theta
_{n+1}\right)  $, essentially, embeds $\overline{W}\backslash S$ into
$\mathbb{CP}_{n+1}$. Then, we decrease $n$ as much as we can by a repeated use
of a Morse-Whitney-Sard lemma~\cite{MoM-VaG1934, WhH1936, SaA1942} which
implies that if $\varphi:M\longrightarrow M^{\prime}$ is a smooth map between
Riemannian manifolds and $\dim_{\mathbb{R}}M<\dim_{\mathbb{R}}M^{\prime}$,
$\varphi\left(  M\right)  $ has measure $0$.

\begin{lemma}
\label{L/ plgmntCPn}There exists $h_{3},...,h_{n+1}\in\mathcal{O}\left(
R\right)  $ such that $S\subset\underset{j\geqslant3}{\cap}\left\{  \partial
h_{j}\neq0\right\}  =\varnothing$ and $\left[  \partial h_{3}:\cdots:\partial
h_{n+1}\right]  $ is an embedding of $\overline{W}$. Moreover, with $\left(
\theta_{j}\right)  _{j\geqslant3}=\left(  \partial h_{j}\right)
_{j\geqslant3}$, $\left[  \theta\right]  =\left[  \theta_{0}:\cdots
:\theta_{n+1}\right]  $ is an almost embedding of $\overline{W}$ such that
$\left[  \theta\right]  \left(  \overline{W}\backslash S\right)  \cap\left[
\theta\right]  \left(  S\right)  =\varnothing$.
\end{lemma}

\begin{proof}
As $\overline{W}$ is compact, (1) enables to find an open covering of
$\overline{W}$ by open subsets $V_{2},...,V_{m}$ of $R$, $\eta_{2},..,\eta
_{m}\in\mathcal{O}\left(  R\right)  $ and $\lambda\in\left\{  0,1\right\}
^{\left\{  2,...,m\right\}  }$ such that each $\frac{\alpha_{j}}%
{\alpha_{\lambda\left(  j\right)  }}$ where $\alpha_{j}=\partial\eta_{j}$ is a
coordinate for $R$ in $V_{j}$. Hence, $\Psi=\left(  \alpha_{0}:\cdots
:\alpha_{m}\right)  $ is an immersion of $\overline{W}$ such that $\Psi\left(
z\right)  =\Psi\left(  z^{\prime}\right)  $ implies $z=z^{\prime}$ when
$\left(  z,z^{\prime}\right)  \in V=\underset{3\leqslant j\leqslant m}{\cup
}V_{j}\times V_{j}$. As for any $z,z^{\prime}\in K=\left(  \overline{W}%
\times\overline{W}\right)  \backslash V$, at least one form in $\beta=\left\{
\alpha_{0},\alpha_{1},\alpha_{0}+\alpha_{1}\right\}  $ don't vanish both at
$z$ and $z^{\prime}$, (2) enables to find families of open subsets $\left(
V_{j}^{\prime}\right)  _{m+1\leqslant j\leqslant r}$ and $\left(
V_{j}^{\prime\prime}\right)  _{m+1\leqslant j\leqslant r}$ of $R$ for which
$\left(  W_{j}\right)  _{m+1\leqslant j\leqslant r}=\left(  V_{j}^{\prime
}\times V_{j}^{\prime\prime}\right)  _{m+1\leqslant j\leqslant r}$ is a
covering of the compactum $K=\left(  \overline{W}\times\overline{W}\right)
\backslash V$, $\eta_{m+1},..,\eta_{r}\in\mathcal{O}\left(  R\right)  $,
$\beta_{m+1},...,\beta_{r}\in\beta$ such that $(\partial_{\beta_{j}}\eta
_{j})\left(  z\right)  \neq(\partial_{\beta_{j}}\eta_{j})\left(  z^{\prime
}\right)  $ for any $\left(  z,z^{\prime}\right)  $ in $W_{j}$. By
construction $\left[  \partial\eta_{0}:\cdots:\partial\eta_{r+1}\right]  $
where $\eta_{r+1}=\eta_{0}+\eta_{1}$ is an injective immersion of
$\overline{W}$ and in the smooth case, the lemma is proved with $\left(
h_{j}\right)  _{3\leqslant j\leqslant n+1}=\left(  \eta_{j}\right)
_{0\leqslant j\leqslant r+1}$.

Lemmas~\ref{L/ plgmntReduc}~and~\ref{L/ plgmTot} applied in the smooth case
gives that for any $k\in\left\{  0,...,r+1\right\}  $ and almost all
$a\in\mathbb{C}^{r+2}$ with $a_{k}=0$, $\left[  \theta_{3}^{a}:\cdots
:\theta_{n+1}^{a}\right]  $ where $\left(  \theta_{j}^{a}\right)
_{j\geqslant3}=\left(  \partial\eta_{j-3}-a_{j-3}\partial\eta_{k}\right)
_{j\geqslant3}$ is an embedding of $\overline{W}$. Let $T=\underset{j\geqslant
3}{\cap}S\cap\left\{  \partial\eta_{j}\neq0\right\}  $. If $q_{\ast}\in
S\backslash T$ and $\partial\eta_{k}\left(  q_{\ast}\right)  \neq0$,
$\theta_{j}^{a}$ don't vanish in $T^{\prime}=T\cup\left\{  q_{\ast}\right\}  $
when $a\notin\underset{m\geqslant3,~q\in T^{\prime}}{\cup}\left\{
t\in\mathbb{C}^{r+2};~t=\partial\eta_{m}\left(  q\right)  /\partial\eta
_{k}\left(  q\right)  \right\}  $. For such an $a$, $\operatorname{Card}%
\underset{j\geqslant3}{\cup}S\cap\left\{  \theta_{j}^{a}=0\right\}
<\operatorname{Card}S\backslash T$ and $\theta^{a}=\left(  \theta_{0}%
,\theta_{1},\theta_{2},\theta_{3}^{a},...,\theta_{n+1}^{a}\right)  $ induces a
well defined canonical map which at $p\in S$ takes the value $\left[
c_{p}^{0}:c_{p}^{1}:c_{p}^{2}:0:\cdots:0\right]  $ where $c_{p}^{\ell
}=\operatorname*{Res}\left(  \theta_{\ell},p\right)  $, $0\leqslant
\ell\leqslant2$. Assume $z\in\overline{W}\backslash S$ and $p\in S$ share the
same image by $\left[  \theta^{a}\right]  $. Let $\ell\in\left\{
0,1,2\right\}  $ such that $c_{p}^{\ell}\neq0$. Then $\theta_{\ell}\left(
z\right)  \neq0$ and because $0\notin\left(  \partial\eta_{j}\right)
_{0\leqslant j\leqslant r+1}\left(  \overline{W}\right)  $, $\left[
\theta^{a}\right]  \left(  z\right)  =\left[  \theta^{a}\right]  \left(
p\right)  $ implies that $\left(  \partial_{\theta_{\ell}}\eta_{k}\right)
\left(  z\right)  \neq0$ and so $a$ belongs to $\left(  \frac{\partial\eta
_{j}}{\partial\eta_{k}}\right)  _{0\leqslant j\leqslant r+1}\left(  \left\{
\partial\eta_{k}\neq0\right\}  \right)  $ which by the Morse-Whitney-Sard
lemma has Lebesgue measure $0$. Repeating a finite number of time this
argument yields $h_{3},...,h_{n+1}\in\mathcal{O}\left(  R\right)  $ satisfying
the lemma's statement if $\left[  \theta\right]  $ where $\theta=\left(
\theta_{0},\theta_{1},\theta_{2},\partial h_{3},...,\partial\theta
_{n+1}\right)  $ is an immersion of $\overline{W}$. Consider $q\in S$, a
coordinate $\zeta$ for $R$ centered at $q$ and for $j\in\left\{
0,..,n+1\right\}  $, write $\theta_{j}=\pi^{\ast}\omega_{j}=(\frac{c_{q}^{j}%
}{\zeta}+f_{j}\circ\zeta)d\zeta$ where $c_{q}^{j}\in\mathbb{C}$ and $f_{j}$ is
holomorphic in a neighborhood of $0$~; as $c$ is nodal, $c_{q}^{\ell}\neq0$
for some $\ell\in\left\{  0,1,2\right\}  $. Then for $j\geqslant3$, near $q$,%
\begin{equation}
\frac{\theta_{j}}{\theta_{\ell}}=\frac{c_{q}^{j}}{c_{q}^{\ell}}+\frac{1}%
{c_{q}^{\ell}}\left[  f_{j}\left(  0\right)  -\frac{c_{q}^{j}}{c_{q}^{\ell}%
}f_{\ell}\left(  0\right)  \right]  \zeta+O\left(  \zeta\right)
\label{F/ pole}%
\end{equation}
and $\left(  \theta_{j}/\theta_{\ell}\right)  _{j\geqslant3}$ extends
holomorphically at $q$ with derivative $\left(  f_{j}\left(  0\right)
/c_{q}^{\ell}\right)  _{j\geqslant3}$. Hence, $\left[  \theta\right]  $ has
rank $1$ at $q$.
\end{proof}

\begin{lemma}
\label{L/ plgmntReduc}Consider $\left(  \omega_{\ell}\right)  _{0\leqslant
\ell\leqslant2}\in K_{3,c}\left(  \Sigma\right)  $ smooth near $\Sigma
\backslash X$, $n\geqslant2$, $\left(  \omega_{j}\right)  _{3\leqslant
j\leqslant n+1}\in K_{n-1,0}\left(  \Sigma\right)  $ and $\theta=\left(
\pi^{\ast}\omega_{j}\right)  _{0\leqslant j\leqslant n+1}$. Assume that
$\left[  \left(  \theta_{\ell}\right)  _{0\leqslant\ell\leqslant2}\right]  $
is well defined, $\left[  \theta\right]  $ is an embedding of $\overline
{W}\backslash S$ and an immersion of $\overline{W}$, $\left[  \theta\right]
\left(  \overline{W}\backslash S\right)  \cap\left[  \theta\right]  \left(
S\right)  =\varnothing$ and that $S\subset\underset{j\geqslant3}{\cap}\left\{
\theta_{j}\neq0\right\}  $. For $a\in\mathbb{C}^{n}$, set $\theta^{a}=\left(
\theta_{0},\theta_{1}-a_{1}\theta_{n+1},...,\theta_{n}-a_{n}\theta
_{n+1}\right)  $.

Then for almost all $a$, $S\subset\underset{j\geqslant3}{\cap}\left\{
\theta_{j}^{a}\neq0\right\}  $, $\left[  \theta^{a}\right]  \left(
W_{0}\backslash S\right)  \cap\left[  \theta^{a}\right]  \left(  S\right)
=\varnothing$, $\left[  \theta^{a}\right]  $ is an immersion of $W_{0}\cup S$
which embeds $\gamma$ and $\left(  \left[  \theta^{a}\right]  \left\vert
_{\overline{W}}\right.  \right)  ^{-1}\left(  \delta_{a}\right)  =\gamma$
where $\delta_{a}=\left[  \theta^{a}\right]  \left(  \gamma\right)  $. If
$n\geqslant3$, $\left[  \theta^{a}\right]  $ is also an embedding of
$W_{0}\backslash S$.
\end{lemma}

\begin{proof}
Because $c$ is nodal $\theta^{a}$ is well defined at $p\in S$ and takes the
value $\left[  c_{p}^{0}:c_{p}^{1}:c_{p}^{2}:0:\cdots:0\right]  $ where
$c_{p}^{j}=\operatorname*{Res}\left(  \theta_{j},p\right)  $, $0\leqslant
j\leqslant n+1$. If $z\in R\backslash S$ and $\theta_{n+1}\left(  z\right)
=0$, $\theta_{a}$ is is well defined at $z$ whatever $a$ is. By the
Morse-Whitney-Sard lemma, $(\frac{\omega_{j}}{\omega_{n+1}})_{0\leqslant
j\leqslant n}\left(  \left\{  \omega_{n+1}\neq0\right\}  \right)  $ has
Lebesgue measure zero. Hence, $\left[  \theta^{a}\right]  $ is well defined
for almost $a\in\mathbb{C}^{n}$.

Set $\tau=\left(  \theta_{j}/\theta_{0}\right)  _{1\leqslant j\leqslant n+1} $
and $\widetilde{\tau}=\left(  \tau_{j}\right)  _{1\leqslant j\leqslant n}$. In
$W_{0}$, the natural affine coordinates of $\theta$ and $\theta^{a}$ are
$\tau$ and $\widetilde{\tau}-\tau_{n+1}a$ where if $p\in S_{0}$, $\tau\left(
p\right)  =\left(  c_{p}^{j}/c_{p}^{0}\right)  _{1\leqslant j\leqslant n+1}$.
The maps considered in \cite{BiE1961, HorL1965L} can be rewritten for Riemann
surfaces in the form%
\begin{align*}
T  &  :\mathbb{C}\times W_{0}\longrightarrow\mathbb{C}^{n+1},~\left(
t,z\right)  \mapsto t\left(  \partial_{\omega_{0}}\tau\right)  \left(
z\right) \\
I  &  :\mathbb{C}\times W_{0}\times W_{0}\longrightarrow\mathbb{C}%
^{n+1},~\left(  t,z,z^{\prime}\right)  \mapsto t\left[  \tau\left(  z\right)
-\tau\left(  z^{\prime}\right)  \right]  .
\end{align*}
$\left[  \theta^{a}\right]  $ is an immersion of $W_{0}^{r}=W_{0}\backslash S$
if and only if $\left(  a,1\right)  \notin T\left(  \mathbb{C}\times W_{0}%
^{r}\right)  $ and is injective on $W_{0}^{r}$ (resp. on $\gamma$) if and only
if $\left(  a,1\right)  \notin I\left(  \mathbb{C}\times W_{0}^{r}\times
W_{0}^{r}\right)  $ (resp. $\left(  a,1\right)  \notin I\left(  \mathbb{C}%
\times\gamma\times\gamma\right)  $). As explained in \cite{HorL1965L}, the
Morse-Whitney-Sard lemma and a homogeneity argument give that for almost all
$a\in\mathbb{C}^{n}$ the first and third above properties are true because
$n\geqslant2$ and that the second is true for almost all $a$ when
$n\geqslant3$.

Using notation of lemma~\ref{L/ plgmntCPn} to write the $\omega_{j}$ in a
coordinate $\zeta$ centered at $q\in S$, (\ref{F/ pole}) applies and gives
that for $j\in\left\{  0,1,2\right\}  \backslash\left\{  \ell\right\}  $%
\[
\frac{\partial\left(  \theta_{j}-\theta_{n+1}a_{j}\right)  /\theta_{\ell}%
}{\partial\zeta}\left(  q\right)  =\frac{1}{c_{q}^{\ell}}\left(  f_{j}\left(
0\right)  -\frac{c_{q}^{j}}{c_{q}^{\ell}}f_{\ell}\left(  0\right)
-f_{n+1}\left(  0\right)  a_{j}\right)
\]
As $\omega_{n+1}\left(  q\right)  \neq0$, it appears that $\left[  \theta
^{a}\right]  $ is regular at $q$ when $a$ don't belong to a finite union of
complex affine hyperplanes of $\mathbb{C}^{n}$.

$S\subset\underset{j\geqslant3}{\cap}\left\{  \theta_{j}^{a}\neq0\right\}  $
is achieved with the trick of lemma's~\ref{L/ plgmntCPn} proof.

If $z\in W_{0}^{r}$ and $p\in S$ share the same image by $\theta_{a}$, $p\in
S_{0}$, $\tau_{n+1}\left(  p\right)  =0$, $\widetilde{\tau}\left(  z\right)
-\tau_{n+1}\left(  z\right)  a=\widetilde{\tau}\left(  p\right)  $ and
$\tau_{n+1}\left(  z\right)  \neq0$ because $\left[  \theta\right]  \left(
z\right)  \neq\left[  \theta\right]  \left(  p\right)  $. Hence, $\left[
\theta^{a}\right]  \left(  W_{0}^{r}\right)  \cap\left[  \theta^{a}\right]
\left(  S\right)  =\varnothing$ if $\left(  a,1\right)  \not \in I\left(
\mathbb{C}\times W_{0}^{r}\times S\right)  $. For such an $a$, $\left(
\left[  \theta^{a}\right]  \left\vert _{\overline{W}}\right.  \right)
^{-1}\left(  \delta_{a}\right)  $ is $(\theta_{a}\left\vert _{\overline
{W_{0}^{r}}}\right.  )^{-1}\left(  \delta_{a}\right)  $ (because
$\gamma\subset W_{0}$) and equals $\gamma$ if and only if $\left(  a,1\right)
\notin I\left(  \mathbb{C}\times W_{0}^{r}\times\gamma\right)  $. Both
conditions are satisfied for almost all $a$ because $n\geqslant2$.
\end{proof}

\begin{lemma}
\label{L/ plgmTot}Hypothesis and notation are those of
lemma~\ref{L/ plgmntReduc}. Then for almost all $a\in\mathbb{C}^{n}$, the
following also hold~:

$\left[  \theta^{a}\right]  \left(  \overline{W}\backslash S\right)
\cap\left[  \theta^{a}\right]  \left(  S\right)  =\varnothing$, $\left[
\theta^{a}\right]  $ is an immersion of $\overline{W}$ which embeds $\gamma$.
If $n\geqslant3$, $\left[  \theta^{a}\right]  $ is an embedding of
$\overline{W}\backslash S$ which induces an embedding of $\overline{X}$ when
$c$ is injective.
\end{lemma}

\begin{proof}
Consider a map $\left[  \theta^{a}\right]  $ obtained in
lemma~\ref{L/ plgmntReduc}, $\alpha\in K\left(  R\right)  $ never vanishing,
$g=\left(  \omega_{j}/\alpha\right)  _{0\leqslant j\leqslant n+1}$ and
$\widetilde{g}=\left(  g_{j}\right)  _{0\leqslant j\leqslant n}$.

Suppose $z\in\overline{W}\backslash S$ and $q\in S$ share the same image by
$\theta^{a}$. Then $z,q\in Z_{0}$ (because $\left[  \theta^{a}\right]  \left(
W_{0}\backslash S\right)  \cap\left[  \theta^{a}\right]  \left(  S\right)
=\varnothing$) and $\widetilde{g}\left(  z\right)  -ag_{n+1}\left(  z\right)
\in\mathbb{C}^{\ast}b_{q}$ where $b_{q}=(0:\operatorname*{Res}\left(
\theta_{1},q\right)  :\operatorname*{Res}\left(  \theta_{2},q\right)
:0:\cdots:0)$, a condition which forces $g_{n+1}\left(  z\right)  \neq0$ since
$\left[  \theta\right]  \left(  z\right)  \neq\left[  \theta\right]  \left(
q\right)  $. Hence, $\left[  \theta^{a}\right]  \left(  \overline{W}\backslash
S\right)  \cap\left[  \theta^{a}\right]  \left(  S\right)  =\varnothing$ when
$a$ isn't in the union of the affine lines $\frac{\widetilde{g}\left(
z\right)  }{g_{n+1}\left(  z\right)  }+\mathbb{C}b_{q}$ where $\left(
q,z\right)  $ takes all values in the finite set $S\times(\overline{W}\cap
Z_{0}\cap\left\{  \theta_{n+1}\neq0\right\}  )$.

If $\left[  \theta^{a}\right]  $ isn't of rank $1$ at a point of $W\cap
Z_{0}\backslash S$, there is $j,k\in\left\{  1,...,n\right\}  $ with $j\neq k$
such that\vspace*{-6pt}%
\[
\left[  g_{j}-a_{j}g_{n+1}\right]  d\left[  g_{k}-a_{k}g_{n+1}\right]
-\left[  g_{k}-a_{k}g_{n+1}\right]  d\left[  g_{j}-a_{j}g_{n+1}\right]
=0\vspace*{-6pt}%
\]
somewhere in $Z_{0}\backslash S$ so that $a\in\underset{z\in Z_{0}}{\cup
}\,\underset{j\neq k}{\cap}H_{z,j,k}$ with\vspace*{-6pt}
\[
H_{z,j,k}=\left\{  t\in\mathbb{C}^{n};~\alpha_{j,k}\left(  z\right)
-t_{j}\beta_{k}\left(  z\right)  -t_{k}\gamma_{j}\left(  z\right)  =0\right\}
\vspace*{-6pt}%
\]
where $\alpha_{j,k}=g_{j}\partial_{\Omega}g_{k}-g_{k}\partial_{\Omega}g_{j}$,
$\beta_{k}=g_{n+1}\partial_{\Omega}g_{k}-g_{k}\partial_{\Omega}g_{n+1}$ and
$\gamma_{j}=g_{j}\partial_{\Omega}g_{n+1}-g_{n+1}\partial_{\Omega}g_{j}$. If
$z\in\overline{W}\cap Z_{0}\backslash S$ is a zero of $\left(  \alpha
_{j,k},\beta_{k},\gamma_{j}\right)  $ for all $j,k\in\left\{  1,...,n\right\}
$ with $j\neq k$ then $\left(  g_{j}\partial_{\Omega}g_{k}-g_{k}%
\partial_{\Omega}g_{j}\right)  \left(  z\right)  =0$ for all $j,k\in\left\{
1,...,n+1\right\}  $ and $\left[  \theta\right]  $ has rank $0$ at $z$. As it
is not the case, $\underset{z\in Z_{0}}{\cup}\,\underset{j\neq k}{\cap
}H_{z,j,k}$ is a finite union of proper subspaces of $\mathbb{C}^{n}$ and
$\left[  \theta^{a}\right]  $ is regular on $\overline{W}\backslash S$ for
almost all $a$.

If $\left[  \theta^{a}\right]  \left(  z\right)  =\left[  \theta^{a}\right]
\left(  z^{\prime}\right)  $ with $z,z^{\prime}\in\overline{W}\backslash S$,
$z,z^{\prime}\in Z_{0}$ and there is $\lambda\in\mathbb{C}^{\ast}$ such that
$\widetilde{g}\left(  z\right)  -\lambda\widetilde{g}\left(  z^{\prime
}\right)  =\left[  g_{n+1}\left(  z\right)  -\lambda g_{n+1}\left(  z^{\prime
}\right)  \right]  a$. Then, either $g_{n+1}\left(  z\right)  =\lambda
g_{n+1}\left(  z^{\prime}\right)  $, $\left[  \theta\right]  \left(  z\right)
=\left[  \theta\right]  \left(  z^{\prime}\right)  $ and hence $z=z^{\prime}$,
either $g_{n+1}\left(  z\right)  \neq\lambda g_{n+1}\left(  z^{\prime}\right)
$ and $a$ belongs to the image $I_{z,z^{\prime}}$ of $\mathbb{C}%
\backslash\left\{  g_{n+1}\left(  z\right)  /g_{n+1}\left(  z^{\prime}\right)
\right\}  $ ($1/0=\infty$ by convention) by\vspace*{-6pt}
\[
H_{z,z^{\prime}}:\lambda\mapsto\left[  \widetilde{g}\left(  z\right)
-\lambda\widetilde{g}\left(  z^{\prime}\right)  \right]  /\left[
g_{n+1}\left(  z\right)  -\lambda g_{n+1}\left(  z^{\prime}\right)  \right]
\vspace*{-6pt}%
\]
When $\left(  z,z^{\prime}\right)  $ belongs to finite set $\Lambda=\left(
Z_{0}\cap\left\{  g_{n+1}\neq0\right\}  \backslash S\right)  ^{2}$ and
$H_{z,z^{\prime}}$ is not constant, $I_{z,z^{\prime}}$ is a holomorphic smooth
curve of $\mathbb{C}^{n}$. Hence $\underset{\left(  z,z^{\prime}\right)
\in\Lambda}{\cup}I_{z,z^{\prime}}$ is of Lebesgue measure $0$. Therefor,
$\left[  \theta^{a}\right]  $ is injective on $\overline{W}\backslash S$ for
almost all $a\in\mathbb{C}^{n}$ if $n\geqslant3$.
\end{proof}

\begin{proof}
[Proof of theorem~\ref{P/ PerturbPlgmnt}]For $a=(\left(  a_{\nu,\ell}\right)
_{0\leqslant\ell\leqslant n-\nu})_{1\leqslant\nu\leqslant n-2}$ in
$\mathbb{C}^{n}\times\mathbb{C}^{n-1}\times\cdots\times\mathbb{C}^{3}$, we set
$\left(  \omega_{0,\ell}\right)  _{1\leqslant\ell\leqslant n+1}=\left(
\omega_{\ell}\right)  _{1\leqslant\ell\leqslant n+1}$ and if $0\leqslant
\nu\leqslant n-2$%
\[
\left(  \omega_{\nu+1,\ell}\right)  _{1\leqslant\ell\leqslant n-\nu}=\left(
\omega_{\nu,\ell}-a_{\nu,\ell}\omega_{\nu,n-\nu+1}\right)  _{1\leqslant
\ell\leqslant n-\nu}.
\]
Lemmas~\ref{L/ 1ereAprox} through \ref{L/ plgmTot} give that $\sigma=\left(
\omega_{0}:\omega_{n-1,1}:\omega_{n-1,2}\right)  $ is an immersion of
$\overline{W}$ which embeds $\gamma$ and satisfies $\left(  \sigma\left\vert
_{\overline{X}}\right.  \right)  ^{-1}\left(  \sigma\left(  \gamma\right)
\right)  =\gamma$. By construction, $\sigma=\left(  \omega_{0}:\omega
_{1}+K_{1}:\omega_{2}+K_{2}\right)  $ where $K_{\ell}=\underset{j\geqslant
3}{\Sigma}b_{\ell,j}\omega_{j}$, each $b_{\ell,j}$ being a universal
polynomial in the coordinates of $a$. Hence $K_{\ell}\left\vert _{\gamma
}\right.  $ can be chosen arbitrarily small in $C^{\infty}\left(  \overline
{X}\right)  $. Moreover, we have $K_{\ell}=\underset{j\geqslant3}{\Sigma
}b_{\ell,j}dh_{j}=dk_{\ell}$ with $k_{\ell}=\underset{j\geqslant3}{\Sigma
}b_{\ell,j}h_{j}$. As $K_{\ell}$ is holomorphic, the function $R_{\ell
}=2\operatorname{Re}k_{\ell}$ is harmonic, is arbitrarily small on $C^{\infty
}\left(  \overline{X}\right)  $ and satisfies $\partial R_{\ell}=K_{\ell}$.
Then $\left(  V_{\ell}\right)  _{_{0\leqslant\ell\leqslant2}}=\left(  U_{\ell
}+R_{\ell}\right)  _{0\leqslant\ell\leqslant2}$ has the expected properties.
\end{proof}

\begin{proof}
[Proof of theorem~\ref{T AEB3}]Applied to a not necessarily exact initial
3-uple $\left(  \omega_{\ell}\right)  _{0\leqslant\ell\leqslant2}$ of forms,
the above lemmas ensure that for arbitrarily small $a$, $\left(  \omega
_{0}:\omega_{n-2,1}:\omega_{n-2,2}:\omega_{n-2,3}\right)  $ an embedding of
$\overline{X}$ when $c$ is injective and that $\sigma$ is an almost embedding
of $\overline{X}$ . Since $\omega_{\nu,\ell}$ has the same singularities and
residues as $\omega_{\ell}$ for any $\ell$, this proves theorem~\ref{T AEB3}
when $\omega$ is smooth near $\Sigma\backslash X$ and hence in the general
case according to preceding reductions.
\end{proof}

\noindent\textbf{Remark. }When $X$ is an open bordered nodal Riemann surface
and $c$ an admissible family, we defined a 4-DN-datum as a 3-uple $\left(
\gamma,u,\theta u\right)  $ where $\gamma=\partial X$ is the oriented boundary
of $X$, $u=\left(  u_{\ell}\right)  _{0\leqslant\ell\leqslant3}\in C^{\infty
}\left(  \gamma\right)  ^{4}$, $\theta u=\left(  \partial\widetilde{u}_{\ell
}^{c_{\ell}}\left\vert _{\gamma}\right.  \right)  _{0\leqslant\ell\leqslant3}$
and $\left[  \left(  \partial\widetilde{u}_{\ell}^{c_{\ell}}\right)  \right]
$ embeds $\overline{X}$ in $\mathbb{CP}_{3}$. Then, a by product of
theorem's~\ref{T AEB2} proof is that $\left\{  u\in C^{\infty}\left(
\gamma\right)  ^{4},~\left(  \gamma,u,\theta u\right)  ~\text{is~a~4-DN-data}%
\right\}  $ is a dense open set of $C^{\infty}\left(  \gamma,\mathbb{R}%
\right)  ^{4}$.

\section{Proofs for the compact case}

Let $Z$ a compact Riemann surface equipped with a K\"{a}hler form $\omega$
such that $\int_{Z}\omega=1$. Let $\ast$ be the usual Hodge operator on forms
and $\delta=-\ast d\ast$ the adjoint of the unbounded operator $d:L_{p,q}%
^{2}\left(  Z\right)  \longrightarrow L_{p,q+1}^{2}\left(  Z\right)  $ where
$L_{r,s}^{2}\left(  Z\right)  $ denotes the space of $\left(  r,s\right)
$-forms with coefficients in $L^{2}$. We denote by $\Delta=\delta d+d\delta$
the Laplace Beltrami operator and by $G$ a Green function for it, that is a
smooth real valued function defined on $Z\times Z$ without its diagonal such
that for very $z\in Z$, the function $G_{z}=G\left(  z,.\right)  $ satisfies
in the sense of currents
\begin{equation}
\Delta G_{z}=\delta_{z}-1.~ \label{F/ lap}%
\end{equation}
where $\delta_{z}$ is the Dirac measure. It is classical (see e.g.
\cite{AuT1998L}) that $G$ is symmetric, $\left\vert G\left(  z,\zeta\right)
\right\vert =O\left(  \ln dist\left(  z,\zeta\right)  \right)  $ for distances
associated to hermitian metrics on $Z$ and that $z\mapsto\int_{Z}G_{z}\omega$
is constant~; as $G$ is unique up to an additive constant, we choose the one
for which this constant is $0$.

As $Z$ is compact, harmonic functions on $Z$ are constant and the
Hodge-De~Rham orthogonal decomposition (see \cite{HoWL1952} or \cite{AuT1998L}%
) of $\varphi\in L^{2}\left(  Z\right)  $ takes the form%
\begin{equation}
\varphi=\mathcal{H}\varphi+\Delta\mathcal{G}\varphi=\mathcal{H}\varphi
+\mathcal{G}\Delta\varphi\label{F/ Hodge}%
\end{equation}
where $\mathcal{H}\varphi=\int_{Z}\varphi\omega$ and $\mathcal{G}\varphi$ is
the function $Z\ni z\mapsto\int_{Z}\varphi G_{z}\omega$.

\begin{lemma}
\label{L/ SingG}Let $z$ be a point of $Z$ and $w$ a coordinate for $Z$
centered at $z$, then $G_{z}-\frac{1}{\pi}\ln\left\vert w\right\vert $ extends
as a smooth function near $z$ and the residue of $G_{z}$ at $z$, that is
$\underset{\varepsilon\longrightarrow0^{+}}{\lim}\frac{1}{2\pi i}%
\int_{\left\vert w\right\vert =\varepsilon}\partial G_{z}$, is equal to $1$.
\end{lemma}

\begin{proof}
It is a plain consequence of the ellipticity of $\Delta$ and of that
$\Delta\left(  G_{z}-\frac{1}{\pi}\ln\left\vert w\right\vert \right)  =-1$ in
a neighborhood of $z$.
\end{proof}

We also need the following lemma which is a minor adaptation of classical
results (see \cite{StE1970L}) about $L_{m}^{p}\left(  Z\backslash S\right)  $,
$\left(  p,m\right)  \in\left[  1,\infty\right]  \times\mathbb{N}$, which is
the Sobolev space of distributions on $Z\backslash\overline{S}$ whose total
differentials up to order $m$ are in $L^{p}\left(  Z\backslash S\right)  $,
$Z$ being equipped with any hermitian metric.

\begin{lemma}
\label{L/ Ext}Let $S$ be a smooth open subset of $Z$. Then there exists an
extension operator $E:C^{\infty}\left(  Z\backslash S\right)  \longrightarrow
C^{\infty}\left(  Z\right)  $ which is continuous from $L_{m}^{p}\left(
Z\backslash S\right)  $ to $L_{m}^{p}\left(  Z\right)  $ for any $\left(
p,m\right)  \in\left[  1,\infty\right]  \times\mathbb{N}$, sends $C^{\infty
}\left(  Z\backslash S,\mathbb{R}\right)  $ to $C^{\infty}\left(  Z\backslash
S,\mathbb{R}\right)  $ and such that $\mathcal{H}\circ E=0$.
\end{lemma}

\begin{proof}
From~\cite{StE1970L}, we get an extension operator $E_{0}$ sending (real
valued) functions on $Z\backslash S$ to (real valued) functions on $Z$ which
is continuous from $L_{m}^{p}\left(  Z\backslash\overline{S}\right)  $ to
$L_{m}^{p}\left(  Z\backslash\overline{S}\right)  $ for any $\left(
p,m\right)  \in\left[  1,\infty\right]  \times\mathbb{N}$. Choose open subsets
$U_{1},V_{1}$ of $Z$ and $\chi_{1},\chi_{2}\in C^{\infty}\left(  Z,\left[
0,1\right]  \right)  $ such that $Z\backslash S\subset\operatorname*{Supp}%
\chi_{1}\subset U_{1}\subset\subset V_{1}$, $\chi_{1}\left\vert _{U_{1}%
}\right.  =1$, $\operatorname*{Supp}\chi_{2}\subset V_{1}\backslash
\overline{U_{1}}$ and $\int_{Z}\chi_{2}\omega=1$. Then, the extension operator
$E$ defined by$\vspace*{-6pt}$
\[
\forall f\in C^{\infty}\left(  Z\backslash S\right)  ,~Ef=\chi_{1}E_{0}%
f-\chi_{2}\int_{Z}\left(  E_{0}f\right)  \chi_{1}\omega\vspace*{-6pt}%
\]
has the same continuousness as $E_{0}$ and for and $f\in C^{\infty}\left(
Z\backslash S\right)  $,$\vspace*{-6pt}$%
\[
\mathcal{H}\left(  E\varphi\right)  =\int_{Z}\chi_{1}E_{0}f\omega-\left(
\int_{Z}\chi_{2}\omega\right)  \int_{Z}\left(  E_{0}f\right)  \omega
=0\vspace*{-12pt}.
\]

\end{proof}

We can now prove theorems~\ref{T/ Runge} and~\ref{T/ compact generique}.

\begin{proof}
[\textbf{Proof of theorem}~\ref{T/ Runge}]Lemma~\ref{L/ RungefaibleLisse}
enable to reduce the proof for functions which are restrictions on
$Z\backslash S$ of harmonic functions on $Z\backslash\overline{S^{\prime}}$
where $S^{\prime}$ is relatively compact smoothly bordered open subset
$S^{\prime}$ of $S$. Let $\varphi:Z\longrightarrow\mathbb{R}$ be harmonic in
$Z\backslash S$ and $\widetilde{\varphi}=E\varphi$ where $E:C^{\infty}\left(
Z\backslash S^{\prime}\right)  \longrightarrow C^{\infty}\left(  Z\right)  $
is an extension operator as in lemma~\ref{L/ Ext}. Consider a family $\left(
S_{\nu}^{\prime}\right)  _{1\leqslant\nu\leqslant N}$ of mutually disjoint
open conformal discs of diameter at most $\varepsilon\in\mathbb{R}_{+}^{\ast}$
such that $\overline{S^{\prime}}=\underset{1\leqslant\nu\leqslant N}{\cup
}\overline{S_{\nu}^{\prime}}$. Using the mean value lemma, one can find for
each $\nu$ $\zeta_{\nu}\in S_{\nu}^{\prime}$ such that$\vspace*{-6pt}$%
\[
\int_{S_{\nu}^{\prime}}\left(  \Delta\widetilde{\varphi}\right)  \omega
=I_{\nu}m_{\nu}%
\]
where $I_{\nu}=\int_{S_{\nu}^{\prime}}\omega$ and $m_{\nu}=\left(
\Delta\widetilde{\varphi}\right)  \left(  \zeta_{\nu}\right)  $. Set
$P_{\varepsilon}=\left\{  \zeta_{\nu},~1\leqslant\nu\leqslant n\right\}  $ and
$Z_{\varepsilon}=Z\backslash P_{\varepsilon}$. The function $\varphi
_{\varepsilon}$ defined on $Z_{\varepsilon}$ by the formula$\vspace*{-6pt}$%
\[
\varphi_{\varepsilon}\left(  z\right)  =%
{\displaystyle\sum\limits_{\nu}}
I_{\nu}m_{\nu}G\left(  \zeta_{\nu},z\right)  \vspace*{-6pt}%
\]
is real analytic on $Z_{\varepsilon}$, real valued provided $\varphi$ is so
and according to~(\ref{F/ lap}) and to the choice of $\left(  \zeta_{\nu
}\right)  $ satisfies$\vspace*{-6pt}$
\[
\Delta\varphi_{\varepsilon}=%
{\displaystyle\sum\limits_{\nu}}
I_{\nu}m_{\nu}\left(  \delta_{\zeta_{\nu}}-1\right)  =%
{\displaystyle\sum\limits_{\nu}}
I_{\nu}m_{\nu}\delta_{\zeta_{\nu}}-\int_{S_{\nu}}\left(  \Delta
\widetilde{\varphi}\right)  \omega\vspace*{-6pt}%
\]
As$^{(}$\footnote{Let us fix $z$ in $Z$ and a geodesic coordinate $w$
centerered at $z$. Then $\omega=idw\wedge d\overline{w}+O\left(  \left\vert
w\right\vert ^{2}\right)  $. Let $f$ be a function of classe $C^{2}$ near $z$.
As $\ast$ acts as multiplication by $-i$, resp. $i$, on forms of bidegree
$(1,0)$, resp. $\left(  0,1\right)  $, and as $\ast\omega=1$, we get $\ast
df=-i\frac{\partial f}{\partial w}dw+i\frac{\partial f}{\partial\overline{w}%
}d\overline{w}$ and $\delta df=-\ast d\ast df=-2\frac{\partial^{2}f}%
{\partial\overline{w}\partial w}+O(\left\vert w\right\vert ^{2})$. On the
other hand, $dd^{c}f=2\frac{\partial^{2}f}{\partial w\partial\overline{w}%
}idw\wedge d\overline{w}$. Evaluation at $z$ yields $\left(  dd^{c}f\right)
_{z}=-\left(  \Delta f\right)  \left(  z\right)  \omega_{z}$.}$^{)}$ $\left(
\Delta\widetilde{\varphi}\right)  \omega=d^{c}d\widetilde{\varphi}$, Stokes
formula yields$\vspace*{-6pt}$
\[
\int_{S^{\prime}}\left(  \Delta\widetilde{\varphi}\right)  \omega
=-\int_{Z\backslash S^{\prime}}d^{c}d\widetilde{\varphi}=0
\]
and it appears that $\varphi_{\varepsilon}$ is harmonic on $Z_{\varepsilon}$.
The singularity of $\varphi_{\varepsilon}$ at $\zeta_{\nu}$ is the same as
$G_{\zeta_{\nu}}$ at $\zeta_{\nu}$ which is a logarithmic isolated one. In
addition,
\[%
{\displaystyle\sum\limits_{\nu}}
\left\vert \operatorname*{Res}\left(  \partial\varphi_{\varepsilon},\zeta
_{\nu}\right)  \right\vert \leqslant%
{\displaystyle\sum\limits_{\nu}}
\int_{S_{\nu}^{\prime}}\left\vert \Delta\widetilde{\varphi}\right\vert
\omega\leqslant Cte\left\Vert \widetilde{\varphi}\right\Vert _{L_{2}%
^{1}\left(  S^{\prime}\right)  }\leqslant Cte\left\Vert \varphi\right\Vert
_{L_{2}^{1}\left(  Z\backslash S\right)  }%
\]

It remains to estimate how $\varphi_{\varepsilon}$ approaches $\varphi$. As
$\mathcal{H}\circ E=0$, the Hodge identity~\ref{F/ Hodge} gives
$\widetilde{\varphi}=\mathcal{G}\Delta\widetilde{\varphi}$. As
$\widetilde{\varphi}=\varphi$ is harmonic in $Z\backslash S^{\prime}$, the
symmetry of $G$ yields that for any $z$ in $Z\backslash S^{\prime}$,
\[
\varphi\left(  z\right)  -\varphi_{\varepsilon}\left(  z\right)
=\int_{S^{\prime}}\left(  \Delta\widetilde{\varphi}\right)  G_{z}\omega-%
{\displaystyle\sum\limits_{\nu}}
\int_{S_{\nu}^{\prime}}\left(  \Delta\widetilde{\varphi}\right)  G\left(
\zeta_{\nu},z\right)  \omega=%
{\displaystyle\sum\limits_{\nu}}
I_{\nu}\left(  z\right)
\]
where $I_{\nu}\left(  z\right)  =\int_{\zeta\in S_{\nu}^{\prime}}\left(
\Delta\widetilde{\varphi}\right)  \left(  \zeta\right)  \left[  G_{z}\left(
\zeta\right)  -G_{z}\left(  \zeta_{\nu}\right)  \right]  \omega\left(
\zeta\right)  $. As $Z\backslash S$ is a compact subset of $Z\backslash
\overline{S^{\prime}}$ and each $S_{\nu}^{\prime}$ has diameter at most
$\varepsilon$, this implies
\begin{equation}
\left\Vert \varphi-\varphi_{\varepsilon}\right\Vert _{0,Z\backslash
S}\leqslant Cte~\varepsilon\left\Vert G\right\Vert _{1,\left(  Z\backslash
S\right)  \times\overline{S^{\prime}}}%
{\displaystyle\sum\limits_{\nu}}
\int_{S_{\nu}^{\prime}}\left\vert \Delta\widetilde{\varphi}\right\vert
\omega\leqslant Cte~\varepsilon\leqslant C_{0}\left\Vert \varphi\right\Vert
_{L_{2}^{1}\left(  Z\backslash S\right)  } \label{F/ majo}%
\end{equation}
where $C_{0}$ depends only of $\left(  S,S^{\prime}\right)  $ and $\left\Vert
G\right\Vert _{1,\left(  Z\backslash S^{\prime}\right)  \times\overline{S}}$
is the supremum norm on $\left(  Z\backslash S\right)  \times\overline
{S^{\prime}}$ of $G$ and its full differential with respect to its second
variable. As $\varphi$ and $\varphi_{\varepsilon}$ are harmonic in
$Z\backslash S^{\prime}$, (\ref{F/ majo}) implies that for any $m$, there is
$C_{m}\in\mathbb{R}_{+}$ depending only of $\left(  S,S^{\prime}\right)  $
such that
\[
\left\Vert \varphi-\varphi_{\varepsilon}\right\Vert _{m,Z\backslash
S}\leqslant C_{m}\varepsilon\left\Vert \varphi\right\Vert _{L_{2}^{1}\left(
Z\backslash S\right)  }.
\]

\end{proof}

\begin{proof}
[\textbf{Proof of theorem~\ref{T/ compact generique}}]We assume without loss
of generality that $S$ is smooth so that we can consider the orientated
boundary $\gamma$ of $Z\backslash S$~; set $u=(\left.  U_{Z,\ell}%
^{a,c}\right\vert _{\gamma})_{0\leqslant\ell\leqslant2}$ and $\theta
u=(\left.  \left(  \partial U_{Z,\ell}^{a,c}\right)  \right\vert _{\gamma
})_{0\leqslant\ell\leqslant2}$. We apply the theorem~\ref{T AEB1} to find in
$C^{\infty}\left(  Z\backslash S\right)  ^{3}$ a triple $\left(  V_{\ell}%
^{1}\right)  $ of harmonic functions on $Z\backslash S $ such that $(V_{\ell
}^{1})$ is arbitrarily close to $\left(  U_{Z,\ell}^{a,c}\right)  $ in
$C^{\infty}\left(  Z\backslash S\right)  ^{3}$ and the canonical $\Phi$ map
associated to $\left(  V_{\ell}^{1}\right)  $ is an almost embedding from
$Z\backslash S$ to $\mathbb{CP}_{2}$ which means in particular that $\Phi$
embeds $\gamma$ into $\mathbb{CP}_{2}$ and $\Phi^{-1}\left(  \delta\right)
=\gamma$ where $\delta=\Phi\left(  \gamma\right)  $. In fact, since $\left(
U_{Z,\ell}^{a,c}\right)  $ is smooth in a neighborhood of $Z\backslash S$, we
are in theorem~\ref{P/ PerturbPlgmnt} situation whose proof concludes that
$\left(  V_{\ell}^{1}-U_{Z,\ell}^{a,c}\right)  =\left(  \operatorname{Re}%
H_{\ell}\right)  $ where $\left(  H_{\ell}\right)  $ is a triple of
holomorphic functions on a neighborhood of $Z\backslash S$.

Consider now $\varepsilon\in\mathbb{R}_{+}^{\ast}$. We apply
theorem~\ref{T/ Runge} to get for each $\ell\in\left\{  0,1,2\right\}  $ a
real valued function $R_{\ell}$ which is harmonic outside a finite subset
$P_{\ell}$ of $S$, has only logarithmic isolated singularities at points of
$P_{\ell}$, whose restriction to $Z\backslash S$ is arbitrarily closed to
$\operatorname{Re}H_{\ell}$ in $C^{\infty}\left(  Z\backslash S\right)  $ and
such that
\begin{equation}
\left\vert \kappa\right\vert _{1}\leqslant C\left\Vert \operatorname{Re}%
H_{\ell}\right\Vert _{L_{2}^{1}\left(  Z\backslash S\right)  } \label{F/ ctrl}%
\end{equation}
where $C$ is some constant and $\kappa=\left(  \kappa_{\ell}\right)
_{0\leqslant\ell\leqslant2}$ where $\kappa_{\ell}$ denotes the family of
residues of $R_{\ell}$ at points of $P_{\ell}$ $0\leqslant\ell\leqslant2$. By
construction, $\left(  V_{\ell}\right)  =\left(  U_{Z,\ell}^{a,c}+R_{\ell
}\right)  $ differs from $U_{Z}^{a,c,p,\kappa}$ only by an additive constant,
is arbitrarily close to $\left(  V_{\ell}^{1}\right)  $ in $C^{\infty}\left(
Z\backslash S\right)  ^{3}$ and, thanks to (\ref{F/ ctrl}), $\left\vert
\kappa\right\vert _{1}\leqslant\varepsilon$ provided $\left(  H_{\ell}\right)
$ is sufficiently small in $C^{\infty}\left(  Z\backslash S\right)  ^{3}$.

When $\left(  V_{\ell}\right)  $ is sufficiently close to $\left(  V_{\ell
}^{1}\right)  $, $\left(  \partial V_{\ell}\right)  $ induces a canonical map
$\Psi$ from $Z\backslash S$ to $\mathbb{CP}_{2}$ which is an immersion and
embeds $\gamma$ into $\mathbb{CP}_{2}$. Since $\Phi$ is an almost embedding,
we can find smoothly bordered open neighborhoods $\Gamma$ and $\Delta$ of
$\gamma$ and $\delta$ in $X$ and $Y=\Phi\left(  X\right)  $ respectively such
that $\Phi\left\vert _{\overline{\Gamma}}^{\overline{\Delta}}\right.  $ is a
diffeomorphism. Set $\Psi\left(  X\right)  =Y^{\prime}$, $\Psi\left(
\overline{\Gamma}\right)  =\Delta^{\prime}$, $\delta^{\prime}=\Psi\left(
\delta\right)  $, consider any hermitian metric on $\mathbb{CP}_{2}$ and
denote by $h$ the associated Hausdorff distance between subsets of $Z$. If
$\left(  V_{\ell}\right)  $ is close enough to $\left(  V_{\ell}^{1}\right)
$, $\Psi\left\vert _{\overline{\Gamma}}^{\Psi\left(  \overline{\Gamma}\right)
}\right.  $ is also a diffeomorphism and $h\left(  Y\backslash\Delta
,Y^{\prime}\backslash\Delta^{\prime}\right)  +h\left(  \delta,\delta^{\prime
}\right)  $ can be made arbitrarily small and in particular less that
$dist\left(  Y\backslash\Delta,\delta\right)  $ which is a positive number
since $\Phi^{-1}\left(  \delta\right)  =\gamma$. This forbids $\Phi
^{-1}\left(  \delta^{\prime}\right)  \cap\Delta^{\prime}\neq\varnothing$.
Hence, $\Phi^{-1}\left(  \delta^{\prime}\right)  =\gamma^{\prime}$ and we can
apply theorem~\ref{P/ Immersion} to conclude that $\Psi$ is an almost
embedding of $Z\backslash S$. Thus, $\left(  a,c,p,\kappa\right)  \in
E_{Z,n}\left(  S\right)  $ where $n=\underset{0\leqslant\ell\leqslant2}{\max
}\operatorname{Card}p_{\ell}$.
\end{proof}

\renewcommand\baselinestretch{1}{\normalsize
\bibliographystyle{amsperso}
\bibliography{ref}
}%

\end{spacing}%

\end{document}